\documentclass{preprint}
\usepackage{amsfonts}
\usepackage{amsmath}
\usepackage{amssymb}
\usepackage{wrapfig}
\usepackage{tikz}
\usepackage{times}
\usepackage{Symbols}
\usepackage{mhequ}
\usepackage{mhenvs}

\newtheorem{assumption}[lemma]{Assumption}

\newcommand{\cM} {{ \mathcal M }} 
\newcommand{\cP} {{ \mathcal P }} 
\newcommand{\cC} {{ \mathcal C }} 

\newcommand{\cQ} {{ \mathcal Q }}
\newcommand{\cT} {{ \mathcal T }}
\newcommand{\PI} {\Pi}
\newcommand{\Norm}[1] {|\!|\!|#1|\!|\!|}
 
\newcommand{\Y } {  \mathbf{Y}  } 
\newcommand{\X } {  \mathbf{X}  } 
\newcommand{\ccdot} {\,\cdot\,} 

\def\CV{C_V}
\def\KV{K_{\!V}}
\def\Law{\mop{Law}}

\newcommand{\push } {_{\hbox{\scriptsize\#}}} 

\def\P{{\mathbf P}}
\def\Q{{\mathbf Q}}

\def\eps{\varepsilon}

\def\E{{\bf E}}

\def\C{{\cal C}}

\def\R{{\bf R}}
\def\dd{{\mathrm{d}}}
\def\e{\,{\mathrm{e}}}
\renewcommand{\one}{\mathbf{1}}

\renewcommand{\Z}{{\mathbb Z}}

\renewcommand{\N}{{\bf N}}

\def\TV{\mathrm{TV}}

\begin{document}

\title{{\bf Asymptotic coupling and a weak form of Harris' theorem
    \\with applications to stochastic delay equations}}

\author{HMS}

\author{ M. Hairer\inst{1,2}, 
  J. C. Mattingly\inst{3}, 
  M. Scheutzow\inst{4}  }

\institute{Mathematics Institute, The University
    of Warwick, Coventry CV4 7AL, UK \\ 
    \email{M.Hairer@Warwick.ac.uk}
     \and 
     Courant Institute, New York University, New York NY 10012, USA\\
    \email{Martin.Hairer@courant.nyu.edu}
    \and
   Department of Mathematics, Duke University,
    Durham NC 27708, USA\\
    \email{jonm@math.duke.edu}
    \and
     Institut f\"ur Mathematik, Sekr.~MA 7-5,
    Fakult\"at II -- Mathematik und Naturwissenschaften, TU Berlin, Str.~des 17.~Juni 136, 10623 Berlin,
    Germany\\
    \email{ms@math.tu-berlin.de}}

\maketitle

\begin{abstract}\parindent1em
  There are many Markov chains on infinite dimensional spaces whose one-step transition kernels are mutually singular when starting from different initial conditions.  We give results which prove unique ergodicity under minimal assumptions on one hand and the existence of a spectral gap under conditions reminiscent of Harris' theorem.

  The first uses the existence of couplings which draw the solutions together as time goes to infinity. Such ``asymptotic couplings'' were central to \cite{EMS,MatNS,H,BM} on which this work builds. As in \cite{BM} the emphasis here is on stochastic differential delay equations.
 
  Harris' celebrated theorem states that if a Markov chain admits a
  Lyapunov function whose level sets are ``small'' (in the sense that
  transition probabilities are uniformly bounded from below), then it
  admits a unique invariant measure and transition probabilities
  converge towards it at exponential speed. This convergence takes
  place in a total variation norm, weighted by the Lyapunov function.

  A second aim of this article is to replace the notion of a ``small set'' by
  the much weaker notion of a ``$d$-small set,'' which takes the
  topology of the underlying space into account via a distance-like
  function $d$. With this notion at hand, we prove an analogue to
  Harris' theorem, where the convergence takes place in a
  Wasserstein-like distance weighted again by the Lyapunov function.

  This abstract result is then applied to the framework of stochastic
  delay equations. In this framework, the usual theory of Harris
  chains does not apply, since there are natural examples for which
  there exist \textit{no} small sets (except for sets consisting of
  only one point). This gives a solution to the long-standing open
  problem of finding natural conditions under which a stochastic delay
  equation admits at most one invariant measure and transition
  probabilities converge to it.
\end{abstract}

\bigskip


\section{Introduction}

There are many Markov chains on infinite dimensional spaces whose one-step transition kernels are mutually singular when starting from 
different initial conditions. Many standard techniques used in the study of Markov chains as exposed for example in \cite{MT} can not be applied to such a singular setting.  In this article, we provide two sets of results which can be applied to general Markov processes even in such a singular 
settings. The first set of results gives minimal, verifiable conditions which are equivalent to the existence of at most one invariant measure.  The second set of results gives a weak version of Harris' theorem which proves the existence of a spectral gap under the existence of a Lyapunov function and a modified ``small set'' condition.

The study of the ergodic theory for stochastic partial differential
equations (SPDEs) has been one of the principal motivations to develop
this theory. While even simple, formally elliptic, linear SPDEs can have transition
probabilities which are mutually singular, the bulk of recent work has
been motivated by equations driven by noise which is ``degenerate''
to varying degrees \cite{EH,BKL,EMS,KS,MatNS,HM,Gap}. The current article
focuses on stochastic delay differential equations (SDDEs) and makes
use of the techniques developed in the SPDE context. That the SPDE
techniques are applicable to the SDDE setting is not surprising since
\cite{EMS} reduced the original SPDE, the stochastic Navier-Stokes
equations, to an SDDE to prove unique ergodicity. In \cite{BM}, the
same ideas were applied directly to SDDEs. There the emphasis was on
additive noise, here we generalize the results to the setting of state
dependent noise. The works \cite{EMS,MatNS,H,BM} all share the central idea
of using a shift in the driving Wiener process to force solutions
starting at different initial conditions together asymptotically as
time goes to infinity.  In \cite{EMS,MatNS,BM}, the asymptotic
coupling was achieved by driving as subset of the degrees of freedom
together in finite time. Typically these were the dynamically unstable
directions, which ensured the remaining degrees of freedom would
converge to each other asymptotically. In \cite{H,HM} the unstable
directions were only stabilized sufficiently by shifting the driving
Wiener processes to ensure that all of the degrees of freedom
converged together asymptotically. This broadens the domain of
applicability and is the tact taken in Section~\ref{sec:absErgodic} to
prove a very general theorem which gives verifiable conditions which
are equivalent to unique ergodicity. In particular, this result
applies to the setting when the transition probabilities are mutually
singular for many initial conditions.

 A simple, instructive example which motivates our discussion  is the following SDDE: 
\begin{equation}\label{sdde}
\dd X(t) = -c X(t)\,\dd t + g(X(t-r))\,\dd W(t)\;,
\end{equation}
where $r > 0$, $W$ is a standard Wiener process, $c > 0$, and $g: \R\to \R$ is a strictly positive, bounded and strictly increasing function.  This can be viewed as a Markov process $\{X_t\}_{t\ge 0}$ on the space $\X = \C([-r,0], \R)$ which possesses an invariant measure for sufficiently large $c$. 
However, in this particular case, given the solution $X_t$ for any $t>0$, the initial condition $X_0 \in \X$ can be recovered \textit{with probability one}, exploiting the law of the iterated logarithm for Brownian motion (see \cite{S05}, Section 2). Thus, if the initial conditions in $\CC([-r,0],\R)$ do not agree, then the transition probabilities for any step of this chain are always mutually singular. In particular, the corresponding Markov semigroup does not have the strong Feller
property and, even worse, the only ``small sets'' for this system are those consisting of one single point. 
The results in Section~\ref{sec:absErgodic} nevertheless apply and allow us to show that \eref{sdde} can have at most one invariant measure and
that converges toward it happens at exponential rate.

While the main application considered in this article is that of stochastic delay equations, the principal theorems are also applicable to a large class of stochastic PDEs driven by degenerate noise. In particular, Theorem~5.4 in \cite{Ergodic} yields a very large class of degenerate SPDEs (essentially semilinear SPDEs with polynomial nonlinearities driven by additive noise, satisfying a H\"ormander condition) for which it is possible to find a contracting distance $d$, see Section~\ref{sec:contractSPDE} below.

\subsection{Overview of main results}

We now summarise the two principal results of this article. The first
is an abstract ergodic theorem which is useful in a number of
different settings and gives conditions equivalent to unique
ergodicity. The second result gives a weak version of Harris' theorem
which ensures the existence of a spectral gap if there exists an
appropriate Lyapunov function.

\subsubsection{Asymptotic coupling and unique ergodicity} 
Let $\X$ be a Polish space with metric $d$ and let
$\X^\infty=\X^{\N_0}$ be the associated space of one-sided infinite
sequences. Given a Markov transition kernel $\CP$ on $\X$, we will
write $\CP_{[\infty]}:\X \to \cM(\X^\infty)$ as the probability kernel
defined by stepping with the Markov kernel $\CP$. Here
$\cM(\X^\infty)$ is the space of probability measures on $\X^\infty$. If $\mu$ is a probability measure on $\X$, then we write $\CP_{[\infty]}\mu$ for the measure in $\cM(\X^\infty)$ defined by  $\int_\X \CP_{[\infty]}(x, \ccdot)\mu(dx)$ .

In general, we will denote by $\cM(\Y)$ the set of probability
measures over a Polish space $\Y$.  Given $\mu_1, \mu_2 \in \cM(\Y)$,
$\cC(\mu_1,\mu_2)$ will denote the set of all couplings of the two
measures. Namely, \begin{equation*} \cC(\mu_1,\mu_2)=\Big\{ \Gamma \in
  \cM(\Y \times \Y) : \PI^{(i)}\push\Gamma = \mu_i \text{ for $i=1,2$}
  \Big\},
\end{equation*}
where $\PI^{(i)}$ is the projection defined by
$\PI^{(i)}(y_1,y_2)=y_i$ and $f\push \mu$ is the push-forward of the
measure $\mu$ defined by $(f\push \mu)(A)=\mu(f^{-1}(A))$. We define
the \textit{diagonal at infinity}
\begin{align*}
  \mathcal{D}=\Big\{ (x^{(1)},x^{(2)}) \in \X^\infty \times \X^\infty:
  \lim_{n\rightarrow \infty} d(x^{(1)}_n ,x^{(2)}_n) =0 \Big\}
\end{align*}
as the set of paths which converge to each other asymptotically. Given
two measures $m_1$ and $m_2$ on $\X^\infty$, we say that $\Gamma \in
\CC(m_1,m_2)$ is an \textit{asymptotic coupling} of $m_1$ and $m_2$ if
$\Gamma(\mathcal{D})=1$.

It is reasonable to expect that if two invariant measure $\mu_1$ and
$\mu_2$ are such that there exists an asymptotic coupling of
$\cP_{[\infty]} \mu_1$ and $\cP_{[\infty]} \mu_2$ then in fact
$\mu_1=\mu_2$. We will see that on the infinite product structure a
seemingly weaker notion is sufficient to prove $\mu_1=\mu_2$.

To this end we define
\begin{equ}[e:defcoupl]
  \widetilde \cC(\mu_1,\mu_2) = \Big\{ \Gamma \in \cM(\Y \times \Y) :
  \PI^{(i)}\push\Gamma \ll \mu_i \text{ for $i=1,2$} \Big\},
\end{equ}
where $ \PI^{(i)}\push\Gamma \ll \mu_i$ means that $\PI^{(i)}\push\Gamma$ is absolutely
continuous with respect to $\mu_i$.

To state the main results of this section, we recall that an invariant
measure $\mu$ for $\cP$ is said to be ergodic if for any invariant
$\phi:\X \rightarrow \R$ ($\cP\phi=\phi$), $\phi$ is $\mu$-almost
surely constant.  
\begin{theorem}
  \label{thm:mainSimpleErgodic} 
  Let $\cP$ be a Markov operator on a Polish space $\X$ admitting two
  ergodic invariant measures $\mu_1$ and $\mu_2$.  The following
  statements are equivalent:
  \begin{enumerate}
  \item \label{mainSimpleErgodic:1} $\mu_1 = \mu_2$.
  \item \label{mainSimpleErgodic:2}There exists an asymptotic coupling of
    $\cP_{[\infty]} \mu_1$ and $\cP_{[\infty]} \mu_2$.
  \item \label{mainSimpleErgodic:3} There exists $\Gamma \in \widetilde \cC(
    \cP_{[\infty]}\mu_1,\cP_{[\infty]}\mu_2)$ such that $\Gamma(\mathcal{D}) > 0$.
  \end{enumerate}
\end{theorem}

\begin{remark}
In \eref{e:defcoupl}, we could have replaced absolute continuity by equivalence. If we have a `coupling'
$\Gamma$ satisfying the current condition, the measure ${1\over 2} \bigl(\Gamma + \cP_{[\infty]} \mu_1\otimes \cP_{[\infty]} \mu_2\bigr)$ satisfies that stronger condition.
\end{remark}

\begin{remark}
  At first, it might seem surprising that it is sufficient to have an
  asymptotically coupling measure which is only \textit{equivalent}
  and not \textit{equal} to the law of the Markov process. However, it
  is important to recall that equivalence on an infinite time horizon
  is a much stronger statement than on a finite one. The key
  observation is that the time average of any function along a typical
  infinite trajectory gives almost surely the value of the integral of the function
  against some invariant measure. This was the key fact used in
  related results in \cite{EMS} and will be central to the proof
  below.
\end{remark}

\begin{remark}
This theorem was formulated in discrete time for simplicity. Since it only
concerns uniqueness of the invariant measure, this is not a restriction since
one can apply it to a continuous time system simply by subsampling it at integer times.
\end{remark}

\subsubsection{A weak version of Harris' Theorem}
\label{sec:harris} 
We now turn to an extension of the usual Harris
theorem on the exponential convergence of Harris chains under a
Lyapunov condition. Recall that, given a Markov semigroup
$\{\CP_t\}_{t \ge 0}$ over a measurable space $\X$, a measurable
function $V \colon \X \to \R_+$ is called a \textit{Lyapunov function}
for $\CP_t$ if there exist strictly positive constants $\CV, \gamma,
\KV$ such that
\begin{equ}
\CP_t V(x) \le \CV e^{-\gamma t} V(x) + \KV\;,
\end{equ}
holds for every $x \in \X$ and every $t \ge 0$. Another omnipresent
notion in the theory of Harris chains is that of a \textit{small
  set}. Recall that $A \subset \CX$ is small for a Markov operator
$\CP_t$ if there exists $\delta > 0$ such that
\begin{equ}[e:smallset]
\|\CP_t(x,\cdot\,) - \CP_t(y,\cdot\,)\|_\TV \le 1-\delta\;,
\end{equ}
holds for every $x,y \in A$.\footnote{In this article, we normalise
  the total variation distance in such a way that mutually singular
  probability measures are at distance $1$ of each other. This differs
  by a factor $2$ from the definition sometimes found in the
  literature.} (This is actually a slightly weaker notion of a small
set than that found in \cite{MT}, but it turns out to be sufficient
for the results stated in this article.) With these definitions at
hand, Harris' theorem states that:

\begin{theorem}
 \label{thm:harrisFirst}
  Let $\{\CP_t\}_{t \ge 0}$ be a Markov semigroup over a measurable
  space $\X$ admitting a Lyapunov function $V$ and a time
  $t_\star > \gamma^{-1} \log \CV$ such that the level sets $\{x \in
  \X\,:\, V(x) \le C\}$ are small for $\CP_{t_\star}$ for every
  $C>0$. Then, there exists a unique probability measure $\mu_\star$
  on $\X$ that is invariant for $\CP_t$. Furthermore, there exist
  constants $\tilde C>0$ and $\tilde \gamma > 0$ such that the
  transition probabilities $\CP_t(x,\cdot\,)$ satisfy
\begin{equ}
  \|\CP_t(x,\cdot\,) - \mu_\star\|_\TV \le \tilde C
  \bigl(1+V(x)\bigr)e^{-\tilde \gamma t}\;,
\end{equ}
for every $t \ge 0$ and every $x \in \X$.
\end{theorem}

\begin{remark}
  The convergence actually takes place in a stronger total variation
  norm weighted by $V$, in which the Markov semigroup then admits a
  spectral gap, see \cite{MT,Harris}.
\end{remark}

While this result has been widely applied in the study of the long-time behaviour of Markov processes \cite{MT}, it does not seem to 
be very suitable for the study of infinite-dimensional evolution equations because the notion of small sets requires that the transition 
measure not be mutually singular for nearby points.

This suggests that one should seek for a version of Harris' theorem
that makes use of a relaxed notion of a small set, allowing for
transition probabilities to be mutually singular. To this effect, we
will introduce the notion of a \textit{$d$-small set} for a given
function $d \colon \X \times \X \to [0,1]$ used to measure distances
between transition probabilities. This will be the content of
Definition~\ref{def:dsmall} below.  If we lift $d$ to the space of
probability measures in the same way that one defines Wasserstein-$1$
distances, then this notion is just \eref{e:smallset} with the total
variation distance replaced by $d$.

However, we can of course not use any distance function $d$ and expect
to obtain a convergence result by combining it simply with Lyapunov
stability.  We therefore introduce the concept of a distance $d$ that
is \textit{contracting} for $\CP_t$ if there exists $\alpha < 1$ such
that
\begin{equ}[e:contractd2]
d \bigl(\CP_t(x,\cdot\,), \CP_t(y,\cdot\,)\bigr) \le \alpha d(x,y)\;,
\end{equ} 
holds for any two points $x, y\in \X$ such that $d(x,y) < 1$. This
seems to be a very stringent condition at first sight (one has the
impression that \eref{e:contractd2} alone is already sufficient to
give exponential convergence of $\CP_t \mu$ to $\mu_\star$ when
measured in the distance $d$), but it is very important to note that
\eref{e:contractd2} is \textit{not} assumed to hold when $d(x,y) =
1$. Therefore, the interesting class of distance functions $d$ will
consist of functions that are equal to $1$ for ``most'' pairs of points
$x$ and $y$. Compare this with the fact that the total variation distance
can be viewed as the Wasserstein-$1$ distance corresponding to the trivial metric
that is equal to $1$ for any two points that are not identical.

With these definitions at hand, a slightly simplified version of our
main abstract theorem states that:

\begin{theorem}\label{thm:main}
  Let $\{\CP_t\}_{t \ge 0}$ be a Markov semigroup over a Polish space
  $\X$ that admits a Lyapunov function $V$.  Assume furthermore that
  there exists $t_\star > \gamma^{-1}\log \CV$ and a lower
  semi-continuous metric $d \colon \X\times \X \to [0,1]$ such that
\begin{claim}
\item $d^2$ is contracting for $\CP_{t_\star}$,
\item level sets of $V$ are $d$-small for $\CP_{t_\star}$.
\end{claim}
Then there exists a unique invariant measure $\mu_\star$ for $\CP_t$ and the convergence 
$d(\CP_t(x,\cdot\,), \mu_\star) \to 0$ is exponential for every $x \in \X$.
\end{theorem}

\subsection{Structure of paper}

In Section~\ref{sec:absErgodic}, we give the proof of
Theorem~\ref{thm:mainSimpleErgodic} as well as a result which under
related hypotheses ensures that the transition probabilities starting from
any point converge to the expected invariant measure. In
Section~\ref{sec:uniqueness}, we apply the results of the preceding
theorems to prove the unique ergodicity and convergence of transition
probabilities for a wide class of SDDE, including those with state
dependent coefficients. In Section~\ref{sec:Harris}, we prove a weak version
of Harris' ergodic theorem which implies exponential convergence in a
type of weighted Wasserstein-$1$ distance on measures and an associated spectral
gap.  In Section~\ref{sec:applgap}, we apply these results to the
SDDE setting under the additional assumption of a Lyapunov
function in order to obtain a spectral gap result. Lastly, in Section~\ref{sec:contractSPDE}, we show how to
apply the results to the SPDE setting, thus providing an alternative proof to the
results in \cite{Gap}.

\vspace{1em}\noindent \textbf{Acknowledgements:} Work on this paper began while
JCM and MS were visiting the Centro di Ricerca Matematica Ennio De
Giorgi in 2006. MH and MS continued the work while both visited the Mittag-Leffler institute and later on during
a workshop at the Oberwolfach mathematical research institute. Lastly, the end of the tunnel was seen while JCM visited the
Warwick Maths Institute. We thank all of these fine institutions for
their hospitality and role in this collaboration. We would also like to thank P. Bubak,  M.S. Kinnally 
and R.J. Williams for pointing out several 
errors in a previous version of the manuscript.

MH gratefully acknowledges support from an EPSRC Advanced Research
Fellowship (grant number EP/D071593/1).  JCM gratefully acknowledges
support from an NSF PECASE award (DMS-0449910) and a Sloan foundation
fellowship. MS gratefully acknowledges support from the DFG Forschergruppe 718 
``Analysis and Stochastics in Complex Physical Systems''.

\section{Asymptotic coupling}
\label{sec:absErgodic}
This section contains the proof of Theorem~\ref{thm:mainSimpleErgodic}
and a number of results which use related ideas. The goal throughout
this section is to use minimal assumptions. We also provide a criterion that yields convergence 
rates toward the invariant measure using a coupling argument. However, in the case of exponential convergence,
a somewhat cleaner statement will be provided by the weak version of Harris' theorem in Section~\ref{sec:Harris}. 
\subsection{Proof of Theorem~\ref{thm:mainSimpleErgodic}: Unique ergodicity through asymptotic coupling}
The proof given below abstracts the essence of the arguments given in
\cite{EMS,BM}. A version of this theorem is presented in
\cite{saintFlourJCM08}. The basic idea is to leverage the fact that if
the initial condition is distributed as an ergodic invariant measure, then
Birkhoff's ergodic theorem implies that the average along a trajectory of a
given function is an almost sure property of the trajectories.
\begin{proof}[of Theorem \ref{thm:mainSimpleErgodic}]
Throughout this proof we will use the shorthand notation $m_i = \cP_{[\infty]} \mu_i$ for $i=1,2$.
  Clearly \ref{mainSimpleErgodic:1}. implies \ref{mainSimpleErgodic:2}. because if $\mu_1
  = \mu_2$, then we can take the asymptotic coupling to be the measure
  $m_1$ (which then equals $m_2$)
  pushed onto the diagonal of $\X^\infty \times \X^\infty$. Clearly
  \ref{mainSimpleErgodic:2}. implies \ref{mainSimpleErgodic:3}. since $\cC(
  \cP_{[\infty]} \mu_1, \cP_{[\infty]} \mu_2) \subset \widetilde \cC(
  \cP_{[\infty]} \mu_1, \cP_{[\infty]} \mu_2)$ and the definition
  of asymptotic coupling implies that there exists a $\Gamma \in \cC(
  \cP_{[\infty]} \mu_1, \cP_{[\infty]} \mu_2)$ with
  $\Gamma(\mathcal{D})=1$. We now turn to the meat of the result,
  proving that \ref{mainSimpleErgodic:3}. implies \ref{mainSimpleErgodic:1}.

  Defining the shift $\theta\colon\X^\infty \rightarrow \X^\infty$ by
  $(\theta x)_k = x_{k+1}$ for $k \in \N_0$, we observe that
  $\theta\push m_i =m_i$ for $i=1,2$. Or in other words $m_i$ is an
  invariant measure for the map $\theta$. In addition, one sees that
  $m_i$ is an ergodic invariant measure for the shift $\theta$ since
  $\mu_i$ was ergodic for $\cP$.

  Fixing any bounded, globally Lipschitz $\phi:\X \rightarrow \R$ we
  extend it to a bounded function $\tilde \phi: \X^\infty \rightarrow \R$ by setting
  $\tilde \phi(x) = \phi(x_0)$ for $x \in \X^\infty$. Now Birkhoff's ergodic
  theorem and the fact that the $m_i$ are ergodic for $\theta$ ensure the
  existence of sets $A_i^\phi \subset \X^\infty$ with $m_i(A^\phi_i)=1$ so
  that if $x \in A_i^\phi$ then
  \begin{align}
    \label{eq:converge}
    \lim_{n \rightarrow \infty} \frac1n \sum_{k=0}^{n-1} \phi(x_k)
    =\lim_{n \rightarrow \infty} \frac1n \sum_{k=0}^{n-1} \tilde\phi(\theta^k x)
    = \int_{\X^\infty} \tilde\phi(x) m_i(\dd x)=\int_\X \phi(z) \mu_i(\dd z)\;.
  \end{align}
  Here, the first and last implications follow from the definition of
$\tilde \phi$. Now let $\Gamma \in \widetilde \cC( m_1,m_2)$ with
  $\Gamma(\mathcal{D}) >0$. Since $\PI^{(i)}\push \Gamma \ll m_i$ and
  both are probability measures we know that $(\PI^{(i)}\push
  \Gamma)(A_i^\phi)= m_i(A_i^\phi)=1$ for $i=1,2$. This in turn implies
  that $\Gamma(\X^\infty \times A_2^\phi)=\Gamma(A_1^\phi \times
  \X^\infty)=1$, so that $\Gamma(A_1^\phi \times
  A_2^\phi)=1$. Hence if $\mathcal{\widetilde D} = \mathcal{D} \cap (
  A_1^\phi \times A_2^\phi )$, we have that $\Gamma(\mathcal{\widetilde D})
  > 0$. In particular, this implies that $\mathcal{\widetilde D}$ is not
  empty. Observe that for any $(x^{(1)},x^{(2)}) \in \mathcal{\widetilde
    D}$ we know that for $i =1,2$
  \begin{align*}
    \lim_{n \rightarrow \infty} \frac1n \sum_{k=0}^{n-1} \phi(x_k^{(i)})=
    \int_\X \phi(z) \mu_i(\dd z)\,.
  \end{align*}
  On the other hand, since $(x^{(1)}, x^{(2)}) \in
  \mathcal{\widetilde D}\subset\mathcal{D}$, we know that $\lim_n
  d(x_n^{(1)},x_n^{(2)})=0$. Combining these facts gives
  \begin{align*}
    \left| \int \phi(x) \mu_1(\dd x) - \int \phi(x) \mu_2(\dd x) \right| &=
    \Bigl| \lim_{n \rightarrow \infty} \Bigl(\frac1n \sum_{k=0}^{n-1}
      \phi(x_k^{(1)}) - \frac1n \sum_{k=0}^{n-1} \phi(x_k^{(2)})\Bigr)\Bigr|\\
    &\leq \lim_{n \rightarrow \infty} \frac{C}{n} \sum_{k=0}^{n-1} d\big(
      x_k^{(1)}, x_k^{(2)}\big) = 0\;,
   \end{align*}
where both the first equality and the second inequality follow from the fact that we choose $(x^{(1)},x^{(2)})$ in $\tilde \CD$.
Therefore,
  \begin{align*}
    \int \phi(x) \mu_1(\dd x) = \int \phi(x) \mu_2(\dd x)
  \end{align*}  
  for any Lipschitz and bounded $\phi$ which implies that $\mu_1=\mu_2$.
\end{proof}

\begin{remark}
Note that it is not true in general that the uniqueness of the invariant measure implies that there exist asymptotic couplings for
any two starting points! See for example \cite{Lindvall,CraWan} for a discussion on the relation between coupling, shift coupling, ergodicity, and mixing.
\end{remark}

\begin{corollary}\label{cor:coupling}
  Let $\cP$ be a Markov operator on a Polish space. If there exists a measurable
  set $A \subset \X$ with the following two properties:
  \begin{claim}
\item   $\mu(A)>0$ for any invariant probability measure $\mu$ of $\CP$,
\item there exists a measurable map $A \times A \ni (x,y) \mapsto \Gamma_{x,y} \in \widetilde \cC( \cP_{[\infty]}\delta_x, \cP_{[\infty]}\delta_y)$
 such that $\Gamma_{x,y}(\mathcal{D}) > 0$ for every $x,y \in A$.
\end{claim}
then $\CP$ has at most one invariant probability measure.
\end{corollary}

\begin{remark}
The measurability of the map $\Gamma$ means that the map $(x,y) \mapsto \int \phi\,d\Gamma_{x,y}$ is
measurable for every bounded continuous function $\phi \colon \X^\infty \times \X^\infty \to \R$.
\end{remark}

\begin{proof}[of Corollary~\ref{cor:coupling}]
Assume that there are two invariant measures $\mu_1$ and $\mu_2$. By the ergodic decomposition, we can assume that $\mu_1$ and $\mu_2$ are both 
ergodic since any invariant  measure can be decomposed into ergodic invariant measures.  We extend the definition of $\Gamma$ to $\X \times \X$ by $\Gamma_{x,y}=\cP_{[\infty]}\delta_x \times \cP_{[\infty]}\delta_y$ for $(x,y) \not \in A\times A$.
For measurable sets $B \subset \X^\infty \times \X^\infty $  define the measure $\Gamma$ by
\begin{align*}
  \Gamma(B) = \int_{\X\times \X} \Gamma_{x,y}(B)\mu_1(dx)\mu_2(dy) 
\end{align*}
and notice that $\Gamma \in  \widetilde\cC( \cP_{[\infty]}\mu_1, \cP_{[\infty]}\mu_2)$ by construction. Furthermore by the assumption that $\Gamma_{x,y}(\mathcal{D}) >0$ for $(x,y) \in A \times A$, we see that $\Gamma(\mathcal{D}) >0$. Hence Theorem~\ref{thm:mainSimpleErgodic}  implies that $\mu_1=\mu_2$.
\end{proof}

\subsection{Convergence of transition probabilities}\label{subs:convergence}

In this section, we give a simple criterion for the convergence of
transition probabilities towards an invariant measure under extremely
weak conditions that essentially state that:
\begin{claim}
\item[1.] There exists a point $x_0 \in \X$ such that the process
  returns ``often'' to arbitrarily small neighbourhoods of $x_0$.
\item[2.] The coupling probability between trajectories starting at
  $x$ and $y$ converges to $1$ as $y \to x$.
\end{claim}
More precisely, we have the following result:

\begin{theorem}\label{thm:conv}
  Let $\CP$ be a Markov kernel over a Polish space $\X$ with metric $d$ that admits an
  ergodic invariant measure $\mu_\star$.  Assume that there exist $B
  \subset \X$ with $\mu_\star(B) > {1\over 2}$ and $x_0 \in \X$ such
  that, for every neighbourhood $U$ of $x_0$ there exists $k>0$ such
  that $\inf_{y \in B} \CP^k(y,U) > 0$. Assume furthermore that there
  exists a measurable map
\begin{equs}
\Gamma \colon \X \times \X &\to \CM(\X^\infty, \X^\infty)\\ 
(y,y') &\mapsto \Gamma_{y,y'}\;, 
\end{equs}
with the property that $\Gamma_{y,y'} \in \cC(
\cP_{[\infty]}\delta_y, \cP_{[\infty]}\delta_{y'})$ for every
$(y,y')$ and such that, for every $\eps > 0$ and every $x\in \X$,
there exists a neighbourhood $U$ of $x$ such that $\inf_{y,y' \in U}
\Gamma_{y,y'}(\CD) \ge 1-\eps$.

Then, $\CP^n(z,\cdot\,) \to \mu_\star$ weakly as $n \to \infty$ for
every $z \in \supp \mu_\star$.
\end{theorem}

\begin{remark}
  This convergence result is valid even in situations where the
  invariant measure is not unique. If the process is irreducible
  however, then $\supp \mu_\star = \X$ and the convergence holds for
  every $z \in \X$ (implying in particular that $\mu_\star$ is
  unique).
\end{remark}

\begin{proof}
We denote by $\CD_N^\eps$ the subset of $\X^\infty \times \X^\infty$given by
\begin{equ}
\CD_N^\eps = \{(X, Y)\,:\, d(X_n, Y_n) \le \eps \;\forall n \ge N\}\;,
\end{equ}
so that $\CD = \bigcap_{\eps>0} \bigcup_{N > 0} \CD_N^\eps$. 

Note first that by the ergodicity of $\mu_\star$ and Birkhoff's
ergodic theorem, there exists a set $A$ with $\mu_\star(A) = 1$ and
such that
\begin{equ}
\lim_{T \to \infty}{1\over T} \sum_{t=0}^T \one_{B}(X_t) = \mu_\star(B) > {1\over 2}\;,
\end{equ}
holds almost surely for every process $X_t$ with $X_0 \in A$.

Fix now $z \in A$ and $\eps>0$ and let $x_0$ be as in the first
assumption of the statement. Fix furthermore a neighbourhood $U$ of
$x_0$ such that $\Gamma_{y,y'}(\CD) \ge 1-\eps$ for $y,y' \in
U$. Define the function $(y,y') \mapsto N_{y,y'}$ by
\begin{equ}[e:defNyy]
N_{y,y'} = \inf \{N > 0 \,:\, \Gamma_{y,y'}(\CD_N^\eps) \ge 1-2\eps\}\;.
\end{equ}
It follows from the choice of $U$ and the fact that $\CD \subset
\bigcup_{N > 0} \CD_N^\eps$ that one has $N_{y,y'} < \infty$ for every
pair $y,y' \in U$. Let now $Y_t$ be a Markov process generated by
$\CP$ with initial distribution $\mu_\star$ and let $Z_t$ be an
independent process with initial condition $z$. Since one has the
bound $\one_B(x)\one_B(z) \ge \one_B(x) + \one_B(z)-1$, it follows
from the definition of $A$ that
\begin{equ}[e:boundB]
  \lim_{T \to \infty}{1\over T} \sum_{t=0}^T
  \one_{B}(Y_t)\one_{B}(Z_t) \ge 2\mu_\star(B) - 1 > 0\;,
\end{equ}
almost surely. Let now $\tau = \inf\{t \ge 0 \,:\, Y_t \in U \;\& \;
Z_t \in U\}$. Before we proceed, let us argue that $\tau$ is almost
surely finite. We know indeed by assumption that there exists $T_U>0$
and $\alpha > 0$ such that $\CP^{T_U}(y,U)\ge \alpha$ for every $y \in
B$. Furthermore, setting $\tau_0 = 0$ and defining $\tau_k$
recursively by
\begin{equ}
  \tau_{k} = \inf\{t \ge \tau_{k-1} + T_U \,:\, Y_t \in B \;\& \; Z_t
  \in B\}\;,
\end{equ}
we have the bound
\begin{equs}
  \P (\tau \ge T) &\le \P (\tau \ge T\,|\, \tau_k < T-T_U)\P(\tau_k < T-T_U) + \P(\tau_k \ge T-T_U) \\
  &\le \P (\tau \neq \tau_1 + T_U\,\&\,\cdots\,\&\, \tau \neq \tau_k + T_U) + \P(\tau_k \ge T-T_U) \\
  &\le (1-\alpha)^k + \P(\tau_k \ge T-T_U)\;,
\end{equs}
where the last inequality follows from the strong Markov
property. Since we know from \eref{e:boundB} that the $\tau_k$ are
almost surely finite, we can make $\P(\tau_k \ge T-T_U)$ arbitrarily
small for fixed $k$ by making $T$ large.

It follows that there exists $T_0> 0$ with $\P(\tau \le T_0) \ge
1-\eps$. Let now $\mu_0$ be the law of the stopped process at time
$T_0$, that is $\mu_0 = \Law (Y_{T_0 \wedge \tau}, Z_{T_0 \wedge
  \tau})$. It follows from the definitions of $T_0$ and $\tau$ that
$\mu_0(U \times U) \ge 1-\eps$.

Since $N_{y,y'}$ is finite on $U\times U$, we can find a sufficiently
large value $T_1>0$ such that
\begin{equ}
\mu_0 \bigl(\{(y,y') \in U\times U \,:\, N_{y,y'} \le T_1\}\bigr) \ge 1-2\eps\;.
\end{equ}
Let now $\tilde \Gamma_{y,z}$ be the coupling between
$\cP_{[\infty]}\mu_\star$ and $\cP_{[\infty]}\delta_z$ obtained by
first running two independent copies $(Y_t, Z_t)$ up to the stopping
time $\tau$ and then running them with the coupling
$\Gamma_{Y_\tau,Z_\tau}$. This is indeed a coupling by the strong
Markov property (recall that we are in the discrete time
case). Setting $T = T_0 + T_1$, it follows immediately from the
construction that under the coupling $\tilde \Gamma_{y,z}$, we have
$\P(d(Y_t, Z_t) \le \eps) \ge 1-4\eps$ for each $n \ge T$, thus
yielding the convergence of $\CP^n(z,\cdot\,)$ to $\mu_\star$ as
required.

Let us now extend this argument to more general starting points and
fix an arbitrary $z \in \supp \mu_\star$ and $\eps>0$. Since $A$ is
dense in $\supp \mu_\star$, it follows from our assumption on $\Gamma$
that there exists $z' \in \X$ such that $\Gamma_{z,z'}(\CD) \ge
1-\eps$. In particular, we can find a time $T_2$ such that
$\Gamma_{z,z'}(\CD_{T_2}^\eps) \ge 1-2\eps$.  Since on the other hand,
we know that $\CP^n(z',\cdot\,) \to \mu_\star$ weakly, it immediately
follows that we can find $T_3 \ge T_2$ and for each $n \ge T_3$ a
coupling $\Gamma$ between $\CP^n(z,\cdot\,)$ and $\mu_\star$ such that
$\Gamma(\{(y,y')\,:\, d(y,y') \le 2\eps\}) > 1-3\eps$, thus concluding
the proof.
\end{proof}

\subsection{More  convergence  results: rates and convergence for all initial conditions}
In this section, we continue to develop the ideas of the previous
sections to show how to obtain the convergence of the transition
probabilities for \textit{every} initial condition and how to obtain a
rate of convergence. We essentially follow the ideas laid out in
\cite{MatNS}. They are sufficent to give exponential convergence, but
not a convergence in any operator norm.  Here we will not attempt to
apply the results to our SDDE setting since
Theorem~\ref{thm:harrisFirst} provides stronger results in our
setting of interest. Nonetheless, the ideas and imagery presented in
this section is useful to build intuition. It is also a useful technique in situations where exponential convergence doesn't hold.

Define the 1-Wasserstein distance for two probability measures $\mu_i$ on $\X$ by 
\begin{align*}
  d_1(\mu_1,\mu_2) = \sup_{f \in \text{Lip}_1} \int f(x) \mu_1(dx)- \int f(x) \mu_2(dx) \;,
\end{align*}
where $\text{Lip}_1$ are the Lipschitz functions $f\colon \X
\rightarrow \R$ with Lipschitz constant one.

Let $\mu_1, \mu_2$ be two probability measures on $\X$ and $\Gamma \in
\CC(\cP_{[\infty]} \mu_1,\cP_{[\infty]} \mu_2)$ and let $Z
\in \X^\infty \times \X^\infty$, with
$Z_n=(Z_n^{(1)},Z_n^{(2)})$, be the stochastic process on $\X
\times \X$ with paths distributed as $\Gamma$. Let
$\mathcal{G}_n$ be the $\sigma$-algebra generated by $Z_n$ and
$\Pi_n$ be the projection on $\X^\infty \times \X^\infty$ defined by
$(\Pi_n x)_k = x_{k+n}$.

We say that $\Gamma \in \CC(\cP_{[\infty]} \mu_1,\cP_{[\infty]}
\mu_2)$ is a \textit{marginally Markovian coupling} if for any stopping time
$\tau$ (adapted to $\mathcal{G}_n$), the conditional distribution $\Law(\Pi_\tau Z\,|\,\CG_\tau)$ of $\Pi_\tau Z$ given $\CG_\tau$
belongs almost surely  to $\CC(\cP_{[\infty]}\delta_{Z^{(1)}_\tau},\cP_{[\infty]}\delta_{Z^{(2)}_\tau})$, 
where $Z_n=(Z_n^{(1)},Z_n^{(2)})$. Notice that
this is weaker than assuming that $Z_n$ is a Markov process.

Let  $\rho : \N \rightarrow (0,\infty)$ be
a strictly decreasing function with $\lim_n \rho(n)=0$. 
We define the ``neighborhood'' of the diagonal
\begin{align*}
  \Delta_\rho= \{ (x^{(1)},x^{(2)}) \in \X^\infty \times \X^\infty :
  d(x_n^{(1)}, x_n^{(2)}) < \rho(n) \}\;.
\end{align*}
For any stochastic process $Z$ on $\X \times \X$ with
$Z_n=(Z_n^{(1)},Z_n^{(2)})$, we define the stopping time  
\begin{equ}[e:deftau]
  \tau_\rho(Z) = \inf\{ n \ge 1\,:\, d(Z_n^{(1)}, Z_n^{(2)}) \geq
  \rho(n) \}
\end{equ}
and for $B \in \X \times \X$ we define the hitting time 
\begin{equ}[e:defsigma]
  \sigma_B(Z)=\inf\{ n \geq 0 : Z_n \in B   \} \ .  
\end{equ}

\begin{theorem}\label{thm:densConverge} 
Consider a Markov operator $\CP$ as before over a Polish space $\X$ with distance function $d \le 1$.

Fix a strictly decreasing rate function $\rho$ and a set
  $B \subset \X \times \X$. Assume that there exists a measurable map
  $z_0 \mapsto \Gamma_{z_0} \in \CC(\cP_{[\infty]}
  \delta_{z_0^{(1)}},\cP_{[\infty]} \delta_{z_0^{(2)}})$ where $z_0=(z^{(1)}_0,z^{(2)}_0) \in \X \times \X$, which is a
  marginally Markovian coupling such that,  if $Z$ with
  $Z_n=(Z_n^{(1)},Z_n^{(2)})$,  is distributed as $\Gamma_{z_0}$ then the
  following assumptions hold:
  \begin{enumerate}
  \item If $z_0 \not \in B$, then $\sigma_B(Z)$ is finite almost surely.
  \item There exists an $\alpha >0$ so that if $z_0\in B$ then
    $\Gamma_{z_0}( \Delta_\rho) \geq \alpha$.
  \end{enumerate}
Then for all $(z_0^{(1)},z_0^{(2)}) \in \X \times \X$,
\begin{align*}
  d_1( \cP_n \delta_{z_0^{(1)}} ,\cP_n \delta_{z_0^{(2)}}) \rightarrow 0 
\quad\text{ as } \quad n \rightarrow 0\ . 
\end{align*}
\end{theorem}
\begin{proof}
  To prove the result, we will construct a new coupling on $\X^\infty
  \times \X^\infty$ from the coupling $\Gamma$. We will do this by
  constructing a process on excursions from $B$. 
  The state space of our process of excursions will be given by: 
  \begin{align*}
    \mathbb{X} =\bigcup_{k=0}^\infty \big( (\X^{k} \times
    \X^{k}) \times B)) \quad \text{and}\quad
    \overline{\mathbb{X}}=\big(\X^\infty \times \X^\infty\big) \cup
    \mathbb{X}\ .
  \end{align*}
  In words, $\mathbb{X}$ is the space of finite (but arbitrary) length trajectories taking values in
  the product space $\X \times \X$ and ending in
  $B$. The space $\overline{\mathbb{X}}$ furthermore contains
  trajectories in $\X \times \X$ of infinite length.

  To build the process on excursions, we begin by constructing a Markov transition
  kernel $Q\colon B \rightarrow\cM(\overline{\mathbb{X}})$ from the
  $\Gamma$'s in the following way. 
  For any $z \in B$, let $Z'$ be a $(\X^\infty \times  \X^\infty)$-valued random variable distributed according 
  to the measure $\Gamma_z$ and set $\tau=\tau_\rho(Z')$ with $\tau_\rho$ as in \eref{e:deftau}. 
  If $\tau = \infty$, then we set $Z = Z'$. If $\tau < \infty$ and $Z'_\tau \in B$, then we set $Z=(Z'_1,\cdots,Z'_\tau)$. 
  Otherwise, let $Z''$ be the $(\X^\infty \times \X^\infty)$-valued random variable distributed according 
  to the measure $\Gamma_{Z'_\tau}$.  Since $\Gamma$ is marginally Markovian and $\tau$ is a 
  stopping time, we know that the law of $\Pi_\tau{Z'_\tau}$ is a coupling of $\cP_{[\infty]}\delta_{Z^{(1)}_\tau},$ 
  and $\cP_{[\infty]}\delta_{Z^{(2)}_\tau}$ almost-surely. Hence we can replace the trajectory of 
  $Z'$ after time $\tau$ with the piece of trajectory $Z''$ and the combined 
  trajectory will still be a coupling starting from the initial data. Setting $\sigma=\sigma_B(Z'')$ (which   is finite almost surely)
  we define $Z$ by $Z=(Z'_1,\cdots,Z_\tau',Z_1'',\cdots,Z_\sigma'')$.
  
  In all cases, $Z$ is either a trajectory of infinite length contained
  in $\Delta_\rho$ or it is a segment of finite length ending in the
  set $B$. Hence $Z$ is a $\overline{\mathbb{X}}$-valued random
  variable and we define $Q_z(\ccdot)$ to be the distribution of $Z$.
  To extend the definition to a kernel $\CQ\colon \mathbb{X}
  \rightarrow\cM(\overline{\mathbb{X}})$, we simply define $\CQ_Z$ for
  $z=(z_1,\cdots,z_k) \in \mathbb{X}$ by $\CQ_Z = Q_{z_k}$ since all
  trajectories in $\mathbb{X}$ terminate in $B$.

  We now construct our Markov process on $\overline{\mathbb{X}}$ which we will
  denote by $\CZ_n=(Z_n^{(1)},Z_n^{(2)})$. We use $\CQ$ which
  was constructed in the preceding paragraph as the Markov
  transition kernel. If ever the segment drawn is of infinite length, then the
  process simply stops.  The only missing element is the initial
  condition. If $z_0 \in B$ then we take $\CZ_0=z_0$. If $z_0 \not \in
  B$, then we take $\CZ_0=(z_0,Z'_1,\cdots,Z'_\sigma)$ where $Z'$ is an
  $\mathbb{X}$-valued random variable distributed according to
  $\Gamma_{z_0}$ and $\sigma=\sigma_B(Z')$.

  Now we define $l_n$ to be the length of $\CZ_n$. Let $n^*=\inf\{n:
  l_n=\infty\}$. Since for all $z \in B$, $\Gamma_z(\Delta_\rho) \geq
  \alpha>0$ the $\P(n^* > n) \leq (1-\alpha)^n$ and thus $n^*$
  is almost surely finite.  Lastly we define $t^*=\sum_{k=0}^{n^*-1}
  l_k$. Since $n^*$, and the $l_k$ are all almost surely finite, one
  sees that $t^*$ is almost surely finite.

  Finally, we are ready to perform the desired calculation. We will
  denote by $Z_t=(Z_t^{(1)},Z_t^{(2)})$ the trajectory in $\X^\infty
  \times \X^\infty$ obtained by concatenating together the segments
  produced by running the Markov chain $\mathcal{Z}_n$ constructed in
  the preceding paragraph.
  For any $f \in \textrm{Lip}_1(\X)$ one has
  \begin{align}\label{eq:convEst}
    \E f(Z_t^{(1)}) - \E f(Z_t^{(2)}) &\leq \E d(Z_t^{(1)},Z_t^{(2)})(\one_{t^* > t/2} + \one_{t^* \leq t/2} )\notag \\
    &\leq \P( t^* > t/2) +  \rho(t/2) \;.
  \end{align}
  Observe that the right hand side is uniform for any $f \in
  \textrm{Lip}_1$.  By assumption $\rho(t/2) \rightarrow 0$ as $t
  \rightarrow \infty$ and since $t^*$ is almost surely finite $\P( t^*
  > t/2) \rightarrow 0$ as $t \rightarrow \infty$.\end{proof}

The following corollary gives a rate of convergence assuming one can
control a appropriate moment of $t^*$. As in \cite{H,MatNS}, this is often done by assuming an appropriate Lyapunov structure. This result \textit{does not} prove convergence in any operator norm. In this way, it is inferior to  Theorem~\ref{thm:harrisFirst} which we prefer. (The fact that the norm does not allow test functions $f \leq V$ can be rectified with more work under additional assumptions.) 
\begin{corollary} In the setting of Theorem~\ref{thm:densConverge},
  let $t^*(z_0)$ be the stopping time defined in the proof of
  Theorem~\ref{thm:densConverge} when starting from
  $z_0=(z_0^{(1)},z_0^{(2)}) \in \X \times \X$. If
  $\E(1/\rho(t^*(z_0))) < \Phi(z_0)$ for some $\Phi :\X\times \X \rightarrow
  (0,\infty)$, then
  \begin{align*}
    d_1( \cP_t \delta_{z_0^{(1)}} ,\cP_t \delta_{z_0^{(2)}}) \leq
    \big(1+ \Phi(z_0)\big) \rho(t/2)
  \end{align*}
\end{corollary}
\begin{proof}
  First observe that the Markov inequality implies that $\P( t^* >
  t/2)= \P(1/\rho(t^*) > 1/\rho(t/2)) \leq \rho(t/2)\E( 1/\rho(t^*))
  \leq \rho(t/2) \Phi(z_0)$. Returning to \eqref{eq:convEst}, we see
  that this estimate completes the proof of the desired result.
\end{proof}

\begin{remark}
For this result to be useful, one needs control over $\E(1/\rho(t^*(z_0)))$. First observe that since $t^*=\sum_{k=0}^{n^*-1} l_k$ and 
$\P( n^* > k) \leq (1-\alpha)^k$ the main difficulty is controlling the appropriate moment of $l_n$. $l_n$ consists of two parts. The 
first is $\tau$, the time to exit $\Delta_\rho$, and the second is $\sigma$ the time to return to $B$. The moments of $\tau$ depend 
on how quickly $\Gamma(\Delta^n_\rho) - \Gamma(\Delta_\rho)$ goes to zero as a function of $n$ where 
$\Delta_\rho^n= \{ (x^{(i)},x^{(2)}) \in \X^\infty \times \X^\infty : d(x_k^{(1)}, x_k^{(2)}) < \rho(k) \text{ for all $k \leq n$}\}$. 
This gives information about how long it takes for $Z$ to leave $\Delta_\rho$ when it is conditioned not to stay in inside 
$\Delta_\rho$ for all times. If this exit time has heavy tails, then it can retard the convergence rate. The return times to $B$ also 
influences the convergence rate. Such a return time is often controlled by a Lyapunov function. 
In the context of obtaining exponential convergence, some of these points are explored in \cite{MatNS,H}.
Subexponential convergence rates via Lyapunov functions has been explored for Harris chains in \cite{MT,vere,BakCatGui08:727,DFG06SG}. 
\end{remark}

 \section{Application of the uniqueness and convergence criteria to
   SDDEs}
\label{sec:uniqueness}
\subsection{Application of the uniqueness criterion to SDDEs}
Fix $r > 0$ and let $\CC = \CC([-r,0],\R^d)$ denote the phase space of a
general finite-dimensional delay equation with delay $r$ endowed with
the sup-norm $\|\cdot\|$.  For a function or a process $X$ defined on
$[t-r,t]$ we write $X_t(s):=X(t+s)$, $s \in [-r,0]$. Consider the
following stochastic functional differential equation:
\begin{equation}\label{e:sdde2}
  \begin{aligned}
    \dd X(t) &= f(X_t)\,\dd t + g(X_t)\,\dd W(t),\\
  X_0 &=\eta\in \mathcal{C}, 
  \end{aligned}
\end{equation}
where $f\colon\C \rightarrow \R^d$ and
$g\colon\C\rightarrow\R^m\times\R^d$ and
$W_t=(W^{(1)}_t,\cdots,W^{(m)}_t)$ is a standard Wiener process. 

We will provide conditions on $f$ and $g$ which ensure that, for every
initial condition $X_0=\eta \in \C$, equation (\ref{e:sdde2}) has a
unique pathwise solution which can then be viewed as a $\C$-valued
strong Markov process. The problem of existence and/or uniqueness of
an invariant measure of such a process (or similar processes) has been
addressed by a number of authors, see for example
\cite{IN,S84,BM,RRG,EGS09}.

While existence of an invariant measure has been proven under natural
sufficient conditions on the functionals $f$ and $g$, the uniqueness
question, as already mentioned in the introduction, has not been
answered up to now even in such simple cases as
\begin{equ}
f(x)=-c \,x(0)\;,\quad g(x)=\psi(x(-r))\;,\qquad d=m=1\;, 
\end{equ}
for some $c>0$ and $\psi$ a strictly positive, bounded and strictly
increasing function \cite{RRG}. One difficulty is that the
corresponding Markov process on $\C$ is not strong Feller. Even worse:
given the solution $X_t$ for any $t>0$, the initial condition $X_0$
can be recovered \textit{with probability one} \cite{S05}.  Another
peculiarity of such equations is that while they do in general
generate a Feller semigroup on $\C$, they often do not admit a
modification which depends continuously on the initial condition --
even if $g$ is linear (see e.g. \cite{M} and \cite{MS}), so that they
do not generate a stochastic flow of homeomorphisms. We do
nevertheless have the following uniqueness result:

\begin{theorem}\label{thm:weakErgodicity}
Assume that $m \ge d > 0$ and that, for every $\eta \in \CC$, $g(\eta)$ admits a right inverse $g^{-1}(\eta)\colon \R^d \to \R^m$.
  If $f$ is continuous and bounded on bounded subsets of $\C$, and for
  some $K\geq 1$, $f$ and $g$ satisfy
\begin{equation*}
\sup_{\eta \in \mathcal{C}} \|g^{-1}(\eta)\| < \infty\;,
\end{equation*}
and 
\begin{equation*}
  2  \scal{f(x) - f(y),x(0)-y(0)}^+ +  \Norm{g(x)-g(y)}^2  \le K \|x-y\|^2\;,
\end{equation*}
where $\Norm{M}^2=\mathrm{Tr}(MM^*)$, then the equation
\eqref{e:sdde2} has at most one invariant measure.
\end{theorem}
\begin{remark}\label{rem:global}
  This assumption is sufficient to ensure the existence of (unique)
  global solutions (see \cite{RS08} for even weaker hypotheses).
\end{remark}
\begin{remark}
  Notice that Theorem~\ref{thm:weakErgodicity} does not ensure that
  there exists an invariant measure only that there exists at most
  one. Even when there is no invariant (probability) measure, the
  ideas in Theorem~\ref{thm:weakErgodicity} can be used to show that
  solutions with nearby initial conditions behave asymptotically the
  same with a positive probability.
\end{remark}

We will use the following two lemmas which are proven at the end of
the section. The first one is similar to Proposition 7.3 in \cite{DZ}.
\begin{lemma}\label{lemma}
  Let $W(t)$, $t \ge 0$ be a standard Wiener process and fix $T>0$ and
  $p>2$.  There exists a function $\rho:[0,\infty) \to [0,\infty)$
  satisfying $\lim_{\lambda \to \infty} \rho(\lambda)=0$ such that the
  following holds: let $Y$ satisfy the equation
\begin{align}\label{eq:lemma}
\dd Y(t)&= -\lambda Y(t) \,\dd t + h(t) \,\dd W(t),\\
Y(0) &= 0,\notag
\end{align}
where $h$ is an adapted process with almost surely c\`adl\`ag sample
paths. Then for any stopping time $\tau$ we have
$$
\E \Bigl( \sup_{0 \le t \le {\tau \wedge T}} |Y(t)|^p \Bigr) \le
\rho(\lambda) \E \Bigl( \sup_{0 \le t \le {\tau\wedge T}} |h(t)|^p
\Bigr).
$$
\end{lemma}

\begin{lemma}\label{lem:contraction}
Let $\lambda > 0$, consider the coupled set of equations: 
  \begin{equs}[2][eq:modifiedSDE]
    \dd X(t) &= f(X_t)\,\dd t + g(X_t)\,\dd W(t)\;,\quad&\quad  X_0& =\eta\;, \\
    \dd \widetilde X(t) &= f(\widetilde X_t)\,\dd t + \lambda (X(t) -
    \widetilde X(t)) \,\dd t+ g(\widetilde X_t)\,\dd W(t)\;, &
    \widetilde X_0 &= \widetilde \eta \;,
  \end{equs}
  and define $Z(t):=X(t)-\widetilde X(t), \; t \ge -r$. Then, for
  every $\gamma_0>0$ there exist $\lambda > 0$ and $C>0$ such that the
  bound $\E \bigl(\sup_{t \ge 0} \e^{\gamma_0 t} \|Z_t\|\bigr)^8 \le (C
  \|Z_0\|)^8$ holds for any pair of initial conditions $X_0$ and
  $\widetilde X_0$.
\end{lemma}

\begin{proof}[of Theorem~\ref{thm:weakErgodicity}]
  We begin by fixing two initial conditions $\eta, \widetilde \eta \in
  \mathcal{C}$ of \eref{eq:modifiedSDE}.
We furthermore define the ``Girsanov shift'' $v$ by
\begin{equ}
v(t) =  \lambda g(\widetilde X_t)^{-1} (X(t) - \widetilde X(t))\;,
\end{equ}
where $\lambda>0$ is chosen as in Lemma \ref{lem:contraction} for
$\gamma_0=1$ (say) and we set $\tau = \inf\{ t \geq 0 : \int_0^t
|v(s)|^2\,\dd s \geq \eps^{-1} \|\eta - \tilde \eta\|^2\}$, where
$\eps>0$ is a small constant to be determined. Thanks to the
non-degeneracy assumption on $g$ and Lemma \ref{lem:contraction}, we
obtain $\lim_{\varepsilon \to 0} \P\{\tau=\infty\}=1$ and
$\lim_{t\to\infty} |X(t) - \tilde X(t)| = 0$ almost surely. In
particular, there exist some $\varepsilon>0$ independent of the initial conditions such that $\P\{\tau=\infty\}>0$.
We will fix such a value of $\eps$ from now on.
  
Setting $\widetilde W(t) = W(t) + \int_0^{t \wedge \tau} v(s) \,\dd
s$, we observe that the Cameron-Martin-Girsanov Theorem implies that
there exists a measure $\Q$ on $\Omega:=\CC([0,\infty),\R^m)$ so that
under $\Q$, $\widetilde W$ is a standard Wiener process on the time
interval $[0,\infty)$.
Let $\overline X$ be the solution of
  \begin{align*}
      \dd \overline X(t) &= f(\overline X_t)\,\dd t + g(\overline X_t)\,\dd \widetilde
      W(t)\qquad \overline X_0=\widetilde \eta\;.
  \end{align*}
  Since $\int_0^{\tau} v^2(s)\,ds \le \eps^{-1}\|\eta - \tilde\eta\|^2$ by construction, 
  the law of $\overline X$ is equivalent on $\CC([0,\infty),\R^d)$ 
  to the law of a solution to
  \eqref{e:sdde2} with initial condition $\widetilde \eta$. This means
  that the law of the pair $(X,\overline X)$ has marginals which are
  equivalent on $\CC([0,\infty),\R^d)$ to solutions to \eqref{e:sdde2} starting respectively
  from $\eta$ and $\widetilde \eta$. 
  Since $\overline X = \widetilde X$ on $\{\tau=\infty\}$, we have
$$
\lim_{t \to \infty} |\overline X(t) -  X(t)|=\lim_{t \to \infty} |\widetilde X(t) - X(t)|=0\; \mbox{ on }\; \{\tau=\infty\} \mbox{ a.s.}
$$
Therefore Corollary~\ref{cor:coupling} implies that the discrete-time chain $(X_{rn})_{n \in \N}$
has at most one invariant probability measure and hence the same is true for the Markov process $(X_t)_{t \ge 0}$.
Since $\CM(\Omega, \Omega)$ endowed with the topology of weak convergence is a Polish space 
\cite{Villani} and since all of the constants appearing in our explicit construction can be chosen
 independently of the initial conditions, the map $(x,y) \mapsto \Gamma_{x,y}$ is indeed measurable. 
\end{proof}

We now give the proof of Lemma~\ref{lemma} which was given at the
start of the section. 

\begin{proof}[of Lemma~\ref{lemma}] We begin by noticing that we need only prove the theorem for the
supremum over a deterministic time interval $[0,T]$. The version over
the random time interval follows by considering the function $\tilde
h(s) = h(s)\one_{[0, \tau\wedge T)}(s)$. Observe that $\tilde h(s)$
again almost surely has c\`adl\`ag paths and if $\tilde Y(t)$ is the
solution to \eqref{eq:lemma} with $h$ replaced by $\tilde h$ then
\begin{align*}
  \sup_{t \leq \tau \wedge T} |Y(s)|^p=   \sup_{t \leq T}
  |\tilde Y(s)|^p \qquad \text{and}\quad  \sup_{t \leq \tau \wedge T}
  |h(s)|^p=   \sup_{t \leq T}
  |\tilde h(s)|^p\,.
\end{align*}
The second identity is clear from the definition of $\tilde h$. The first
follows from the observation that $Y(s)=\tilde Y(s)$ for $s \leq \tau
\wedge T$ and, for $s > \tau \wedge T$, $|\tilde Y(s)|$ only decreases since
$\tilde h$ is identically zero.  Hence it is enough to prove the lemma
over a deterministic time interval. 

We begin by observing that the solution $Y$ can be represented in the
form
\begin{equation}\label{voc}
Y(t) = \e^{-\lambda t} \int_0^t \e^{\lambda s} h(s) \,\dd W(s)\;.
\end{equation}
Therefore, using Burkholder's inequality and abbreviating
$h^*:=\sup_{0 \le t \le T} |h(t)|$, we obtain
\begin{equs}
  \E|Y(t)|^p&=\e^{-\lambda t p}\, \E \Big| \int_0^t \e^{\lambda s} h(s) \,\dd W(s) \Big| ^p \\
  &\le C_p \e^{-\lambda t p}\, \E \Big| \int_0^t \e^{2\lambda s} h^2(s) \,\dd s \Big|^{p/2} \label{est}\\
  &\le C_p \E (h^*)^p (2\lambda)^{-p/2}.
\end{equs}
Let $N \in \N$ and define $t_k:=t_k(N):=kT/N$  for $k=0,...,N$ and
$$
I_k(t):=\int_{t_k}^{t} h(s) \,\dd W(s),\,t_k \le t \le t_{k+1},\,k=0,...,N-1.
$$
Notice that $I_k(t)$ is a local martingale with respect to the
filtration it generates.  Integrating (\ref{voc}) by parts, we get
$$
Y(t)=Y(t_k)\e^{-\lambda(t-t_k)} + I_k(t) - \lambda \int_{t_k}^t 
\e^{-\lambda(t-s)} I_k(s) \,\dd s,\qquad t_k \le t \le t_{k+1}.
$$
Hence,
\begin{eqnarray*}
  \sup_{0 \le t \le T} |Y(t)|^p &=& \max_{k=0,...,N-1} \sup_{t_k \le t \le t_{k+1}} |Y(t)|^p\\
  &\le& 2^{p-1}  \max_{k=0,...,N-1} \left( |Y(t_k)|^p + 2  \sup_{t_k \le t \le t_{k+1}} |I_k(t)|^p\right).
\end{eqnarray*}
Using Burkholder's inequality and (\ref{est}), we get
\begin{eqnarray*}
  \E \sup_{0 \le t \le T} |Y(t)|^p &\le& 2^{p-1} N C_p \E (h^*)^p (2 \lambda)^{-p/2} + 
  2^p N \max_{k=0,...,N-1} \E  \sup_{t_k \le t \le t_{k+1}} |I_k(t)|^p\\
  &\le& 2^{p-1} C_p \E (h^*)^p \left( N (2 \lambda)^{-p/2} + 2 T^{p/2} N^{1-\frac{p}{2}} \right).
\end{eqnarray*}
For each $\varepsilon>0$, we can choose $N$ large enough such that the
coefficient of $\E (h^*)^p$ becomes smaller than $\varepsilon$ for all
sufficiently large $\lambda$.
\end{proof}

We conclude with the proof of Lemma~\ref{lem:contraction}:

\begin{proof}[of Lemma~\ref{lem:contraction}]
 First observe that the pair of equations
 \eqref{eq:modifiedSDE} admits a unique global solution (see
 \cite{RS08}).  Setting $Z(t)= X(t)-\widetilde X(t)$, we see that
 \begin{align*}
   \dd |Z(t)|^2 &= 2\scal{f(X_t)- f(\widetilde X_t),Z(t)}\,\dd t +
   \Norm{g(X_t)-g(\widetilde X_t)}^2 \,\dd t - 2\lambda  |Z(t)|^2 \,\dd t + \dd M(t)\\
   & \leq K \|Z_t\|^2\,\dd t - 2\lambda |Z(t)|^2 \,\dd t + \dd M(t)
 \end{align*}
 where $M(0)=0$ and $\dd M(t) = 2\scal{Z(t), (g(X_t) - g(\widetilde
   X_t))\,\dd W(t)}$. Define now $Y(t)= \e^{\alpha t}|Z(t)|^2$ for a constant $\alpha$ to de determined later. 
    Then
  \begin{align*}
     \dd Y(t) &= \alpha Y(t)\,\dd t + \e^{\alpha t}\,\dd |Z(t)|^2\\
          &\leq (\alpha-2\lambda) Y(t)\,\dd t+  K \e^{\alpha t}\|Z_t\|^2\,\dd t  +  \e^{\alpha t}\,\dd M(t)\\
          &\leq  (\alpha-2\lambda) Y(t)\,\dd t+  K \e^{\alpha r}\|Y_t\|\,\dd t  +  \e^{\alpha t}\,\dd M(t)\;.
 \end{align*}
 Setting   $N(t)=  \int_0^t
  \e^{-\lambda (t-s)} \e^{\alpha s}\,\dd M(s)$ and $\kappa = 2\lambda - \alpha$,  the  variation of constants formula thus yields
 \begin{equs}
  Y(t) &\leq \e^{-\kappa t}Y(0) + K e^{\alpha r}\int_0^t
 \e^{-\kappa (t-s)}\|Y_s\|\,\dd s + N(t)\\
 &\le  \e^{-\kappa t}Y(0) + {K e^{\alpha r}\over \kappa} \sup_{s \in [0,t]} \|Y_s\| + N(t)\;.
    \end{equs}
For $\eps > 0$, let now $\tau_\eps$ be the stopping time defined by $\tau_\eps = 2r \wedge \inf\{t \ge 0\,:\, \|Y_t\| \ge \eps^{-1}\}$.
It follows that there exists a constant $\bar K$  independent 
of $\alpha$, $\lambda$ and $\eps$ such that
\begin{equs}
 \E\sup_{s \in [0,\tau_\eps]} \|Y_s\|^4 &\leq \bar K \|Y_0\|^4 + \frac{\bar K e^{4\alpha r}}{\kappa^4}\E \sup_{s \in [0,\tau_\eps]} \|Y_s\|^4  + \bar K\E \sup_{s \in[0,\tau_\eps]} |N(s)|^4\;.
\end{equs}
Now observe that by Lemma~\ref{lemma}, we have for $N$ the bound
\begin{equ}
   \E \sup_{s \in[0,\tau_\eps]} |N(s)|^4 \leq 
\rho(\lambda)\,2^4 \E \sup_{s \in [0,\tau_\eps]}\big(\e^{4\alpha
 s}|Z(s)|^4\Norm{g(X_s)-g(\widetilde X_s)}^4\big) \leq  C e^{2\alpha r} \rho(\lambda) \E \sup_{s \in [0,\tau_\eps]}  \|Y_s\|^4 \;,
\end{equ}
for a constant $C$  independent of $\alpha$ and $\lambda$.
This shows that we can find a function $\alpha \mapsto \Lambda(\alpha)$ such that  
both $\bar K e^{4\alpha r}/(\Lambda(\alpha)-\alpha)^4\le {1\over 4}$ and 
$\bar K C e^{\alpha r}\rho(\Lambda(\alpha)) \le {1\over 4}$, thus obtaining
\begin{equ}
 \E\sup_{s \in [0,\tau_\eps]} \|Y_s\|^4 \leq \bar K \|Y_0\|^4 + {1\over 2}  \E\sup_{s \in [0,\tau_\eps]} \|Y_s\|^4\;,
\end{equ}
provided that we choose $\lambda = \Lambda(\alpha)$.
Since this bound is independent of $\eps> 0$, we can take the limit $\eps \to 0$, so that the monotone convergence theorem yields $ \E\sup_{s \in [0,2r]} \|Y_s\|^4 \le 2 \bar K \|Y_0\|^4$. In terms of our original process $Z$, we conclude that 
\begin{align}\label{eq:boundZ}
 \E \|Z_r\|^8 \leq  2 \bar K \|Z_0\|^8\qquad\text{and}\qquad    \E \|Z_{2r}\|^8 \leq  2 \bar K e^{-4\alpha r} \|Z_0\|^8\,.
\end{align}
Since $\bar K$ is independent of $\alpha$, we can ensure that $2 \bar K e^{-4\alpha r} \leq e^{-19r \gamma_0}$
by taking $\alpha$ (and therefore also $\lambda$) sufficiently large.
Iterating \eref{eq:boundZ}, we obtain
\begin{equ}[e:iter]
  \E \|Z_{2nr}\|^8 \leq   e^{-18r \gamma_0 n }\|Z_0\|^8 \qquad\text{and}\qquad  \E \|Z_{(2n+1)r}\|^8 \leq  2 \bar K  e^{-18r \gamma_0 n } \|Z_0\|^8\,.
\end{equ}
Note now that if $t \in [nr,(n+1)r]$, then $\|Z_t\| \leq \|Z_{nr}\| + \|Z_{(n+1)r}\|$. Therefore, there exists a constant $C>0$ such that
\begin{align*}
  \sup_{t \geq 0} e^{8\gamma_0 t} \|Z_t\|^8 \leq  C\sum_{n=0}^\infty e^{8\gamma_0 rn} \|Z_{rn}\|^8\,.
\end{align*}
Hence using \eqref{e:iter}, we have for a different constant $C>0$
\begin{align*}
\E  \sup_{t \geq 0} e^{8\gamma_0 t} \|Z_t\|^8 \leq  C  \|Z_{0}\|^8  \sum_{n=0}^\infty e^{-\gamma_0 rn}\,.
\end{align*}
Since the sum on the right hand side converges, the proof is complete. 
\end{proof}

This shows the uniqueness of the invariant measure for a large class
of stochastic delay equations.  It turns out that under exactly the
same conditions, we can ensure that the invariant measure is not only
unique, but that transition probabilities converge to it.

\subsection{Convergence of transition probabilities of SDDEs}

In this section we will apply the abstract results of Section
\ref{subs:convergence} to the $\CC$-valued Markov process $(X_t)$ which
we introduced in the previous section. We will denote its transition
probabilities by $\cP_t(\eta,.)$. We will prove the following result:

\begin{theorem}\label{thm:convSDDE}
  Let the assumptions of Theorem~\ref{thm:weakErgodicity} be
  satisfied. If the Markov process $X_t$, $t\ge 0$ admits an invariant
  probability measure $\mu$, then for each $\eta \in \C$ we have
  $\cP_t(\eta,\cdot\,) \to \mu$ weakly.
\end{theorem}

We start with the following lemma (which we will also need in Section \ref{sec:applgap}):

\begin{lemma}\label{lem:support}
  Let the assumptions of Theorem~\ref{thm:weakErgodicity} be satisfied
  and denote by $\mathrm{B}_R$ the closed ball in $\C$ with radius $R$
  and center 0. Then for each $R,\delta >0$ and each $t_\star \ge 2r$,
$$
\inf_{\|\eta\| \le R} \cP_{t_\star} \big(\eta,\mathrm{B}_{\delta} \big) > 0.
$$
\end{lemma}

\begin{proof}
  The proof resembles that of Lemma 2.4 in \cite{SS02}.  Fix $R \ge
  \delta >0$. For each $y \in {\bf R}^d$, $|y| \le 3R/2$, let
  $h=h_y:[0,t_\star] \to {\bf R}^d$ be continuously differentiable
  with Lipschitz constant at most $2R/r$ and satisfy $h=0$ on
  $[r,t_\star]$, $h(0)=y$. Define
$$
D(t):= |X(t)-h(t)|^2 - (\delta/2)^2,
$$ 
where $X$ solves SDDE \eqref{e:sdde2} with initial condition $\eta \in
\C, \|\eta\| \le R$, $y:=\eta(0)-(\delta/2,0,...,0)^T$ and $h=h_y$ is
defined as above.  Then
$$
\dd D(t) = 2\scal{X(t)-h(t),f(X_t) - h'(t)} \,\dd t +
2\scal{X(t)-h(t),g(X_t)\,\dd W(t)} + \sum_{i,j} g_{ij}^2(X_t) \,\dd t,
$$
while $D(0)=0$. Let $\tau:=\inf\{t \ge 0: |D(t)| >
(\delta/4)^2\}$. Let now $W_1$ be a Wiener process that is independent
of $W$ and set
$$
Y(t):= D(t \wedge \tau) + (W_1(t)-W_1(\tau)) {\bf 1}_{t \ge \tau}\;.
$$
This is a semimartingale with $Y(0)=0$ which fulfills the conditions
of Lemma I.8.3 of \cite{Bass} (with $(\delta/4)^2$ in place of
$\eps$). Therefore, there exists $p>0$ such that for all $\|\eta\| \le
R$ we have
$$
\cP_{t_\star} \big(\eta,\mathrm{B}_{\delta} \big) \geq 
 \P \big(\sup_{0 \le t \le t_\star}  |Y(t)| \le (\delta/4)^2|X_0=\eta\big) \ge p\;,
$$
thus concluding the proof.
\end{proof}

\begin{proof}[of Theorem~\ref{thm:convSDDE}]
  Theorem~\ref{thm:conv} is formulated for discrete time, so we first
  show that the two conditions are satisfied for the Markov kernel
  $\cP_{t}$ for some $t>0$. The previous lemma immediately implies
  that the first condition of Theorem~\ref{thm:conv} is satisfied for
  any $t>0$ and any sufficiently large $k$. The second condition
  follows from the fact that there exists (a small value) $\delta > 0$
  such that the metric $d(x,y):=1 \wedge \delta^{-1} \|x-y\|$ on $\C$
  is contracting for $\cP_{t}$ (see Definition~\ref{def:contr}) for
  any sufficiently large $t>0$ (which is proved in
  Section~\ref{sec:contrSDDE}) and Proposition~\ref{prop:contr}.
  Since the support of $\mu$ equals $\C$, it therefore follows from
  Theorem~\ref{thm:conv} that for some suitable $t$, all transition
  probabilities of the chain associated to $\cP_{t}$ converge to $\mu$
  weakly. To show that even all transition probabilities of the
  continuous-time Markov process $(X_t)$ converge to $\mu$ weakly, it
  suffices to observe that (by Proposition \ref{prop:apriori}) there
  exists a constant $C$ such that $d(\cP_\tau \nu,\cP_\tau \tilde\nu)
  \leq C \,d(\nu,\tilde\nu)$ for all $\tau \in [0,t]$ and all
  $\nu,\tilde \nu$.
\end{proof}

\begin{remark}
It is not true in general that one has exponential convergence under the 
assumptions of Theorem~\ref{thm:weakErgodicity} (plus the existence of an invariant measure)
alone. Consider for example the one-dimensional SDE
\begin{equ}
dx = - {x \over (1+x^2)^\alpha}\,dt + dW(t)\;,
\end{equ}
then it is known that for $\alpha \in ({1\over 2},1)$ it has a unique invariant measure, but that convergence
of transition probabilities is only stretched exponential \cite{NonExist}. However, it does satisfy the one-sided
Lipschitz condition of Theorem~\ref{thm:weakErgodicity}.
\end{remark}

\section{A weak form of Harris' theorem}
\label{sec:Harris}
In this section, we show that under very mild additional assumptions
on the dynamic of \eref{e:sdde2}, the uniqueness result for an
invariant measure obtained in the previous section can be strengthened
to an exponential convergence result in a type of weighted Wasserstein
distance. Our main ingredient will be the existence of a
\textit{Lyapunov} function for our system. Recall that a Lyapunov
function for a Markov semigroup $\{\CP_t\}_{t\ge 0}$ over a Polish
space $\X$ is a function $V\colon \X \to [0,\infty]$ such that $V$ is
integrable with respect to $\CP_t(x,\cdot\,)$ for every $x\in \X$ and
$t\ge 0$ and such that there exist constants $\CV, \gamma, \KV>0$ such
that the bound
\begin{equ}[e:Lyapunov]
\int V(y)\,\CP_t(x,\dd y) \le \CV \e^{-\gamma t} V(x) + \KV\;,
\end{equ}
holds for every $x \in \X$ and $t \ge 0$. In the usual theory of
stability for Markov processes, the notion of a ``small set'' plays an
equally important role. We say that a set $A \subset \X$ is
\textit{small} if there exists a time $t>0$ and a constant $\eps > 0$
such that
\begin{equ}[e:SGcond]
\|\CP_t(x,\cdot\,)-\CP_t(y,\cdot\,)\|_\TV \le 1 - \eps\;,
\end{equ}
for every $x,y \in A$. Recall that the total variation distance
between two probability measures is equal to $1$ if and only if the
two measures are mutually singular. A set is therefore small if the
transition probabilities starting from any two points in the set have
a ``common part'' of mass at least $\eps$. The classical Harris
theorem \cite{MT,Harris} then states that:

\begin{theorem}[Harris]
  Let $\CP_t$ be a Markov semigroup over a Polish space $\X$ such that
  there exists a Lyapunov function $V$ with the additional property
  that the level sets $\{x\,:\, V(x) \le C\}$ are small for every
  $C>0$. Then, $\CP_t$ has a unique invariant measure $\mu_\star$ and
  $\|\CP_t(x,\cdot\,) - \mu_\star\|_\TV \le C \e^{-\gamma_\star t} (1+
  V(x))$ for some positive constants $C$ and $\gamma_\star$.
\end{theorem}

The proof of Harris' theorem is based on the fact that a semigroup
satisfying these assumptions has a spectral gap in a modified total
variation distance, where the variation is weighted by the Lyapunov
function $V$. This theorem can clearly not be applied to Markov
semigroups generated by stochastic delay equations in general.  As
already mentioned earlier, it is indeed known that even in simple
cases where the diffusion coefficient $g$ only depends on the past of
the process, the initial condition can be recovered exactly from the
solution at any subsequent time. This implies that in such a case
\begin{equ}
\|\CP_t(x,\cdot\,)-\CP_t(y,\cdot\,)\|_\TV = 1
\end{equ}
for every $x \neq y$ and every $t>0$, so that \eref{e:SGcond}
fails. We would therefore like to replace the notion of a small set
\eref{e:SGcond} by a notion of ``closedness'' between transition
probabilities that reflects the topology of the underlying space $\X$.
Before we state our modified notion of a $d$-small set, we introduce
another notation: given a positive function $d\colon \X \times \X \to
\R_+$, we extend it to a positive function $d\colon \CM_1(\X) \times
\CM_1(\X) \to \R_+$, where $\CM_1(\X)$ stands for the set of (Borel)
probablity measures on $\X$, by
\begin{equ}[e:liftd]
  d(\mu,\nu) = \inf_{\pi \in \CC(\mu,\nu)} \int_{\X^2} d(x,y)\,\pi(\dd
  x,\dd y)\;.
\end{equ}
If $d$ is a metric, then its extension to $\CM_1(\X)$ is simply the
corresponding Wasserstein-$1$ distance. In this section, we will be
considering functions $d\colon \X\times \X \to \R_+$ that are not
necessarily metrics but that are ``distance-like'' in the following
sense:

\begin{definition}
  Given a Polish space $\X$, a function $d\colon \X\times \X \to \R_+$
  is distance-like if it is symmetric, lower semi-continuous, and such
  that $d(x,y) = 0 \Leftrightarrow x = y$.
\end{definition} 

Even though we think of $d$ as being a kind of metric, it need not
satisfy the triangle inequality. However, when lifted to the space of
probability measures, $d$ provides a reasonable way of measuring
distances between measures in the sense that $d(\mu, \nu) \ge 0$ and
$d(\mu,\nu) = 0 \Leftrightarrow \mu = \nu$, the latter property being
a consequence of the lower semi-continuity of $d$. The lower
semicontinuity of $d$ also ensures that the infimum in \eref{e:liftd}
is always reached by some coupling $\pi$.  With this notation at hand,
we set:

\begin{definition}\label{def:dsmall}
Let $\CP$ be a Markov operator over a Polish space $\X$ endowed with a distance-like 
function $d \colon \X \times \X \to [0,1]$.
A set $A \subset \X$ is said to be $d$-small if there exists $\eps > 0$ such that 
\begin{equ}[e:coupling]
d\bigl(\CP(x,\cdot\,),\CP(y,\cdot\,)\bigr) \le 1-\eps\;,
\end{equ}
for every $x,y \in A$.
\end{definition}

\begin{remark}
 If $d(x,y) = d_\TV(x,y) := \one_{x \neq y}$, then the notion of a $d$-small set coincides with the notion of a small
set given in the introduction, since $\|\mu-\nu\|_\TV =  d_\TV(\mu,\nu)$.
\end{remark}

In general, it is clear that having a Lyapunov function $V$ with $d$-small level sets cannot be sufficient
to imply the unique ergodicity of a Markov semigroup. A simple example is given by the Glauber dynamic of the
2D Ising
model which exhibits two distinct ergodic invariant measures at low
temperatures, but for which every set is $d$-small if $d$ is a distance function that
metrises the product topology on the state space $\{0,1\}^{\Z^2}$, for
example $d(\sigma, \sigma') = \sum_{k \in \Z^2} {1\over
  2^{|k|}}|\sigma_k - \sigma_k'|$.

This shows that if we wish to make use of the notion of a $d$-small set, we should impose additional assumptions
on the function $d$. One feature that distinguishes the total variation distance $d_\TV$ among other distance-like functions is that,
for any Markov operator $\CP$, one \textit{always} has the contraction property
\begin{equ}
d_\TV (\CP\mu, \CP\nu) \le d_\TV(\mu,\nu)\;.
\end{equ}
It is therefore natural to look for distance-like functions with a similar property. This motivates the following definition:

\begin{definition}\label{def:contr}
Let $\CP$ be a Markov operator over a Polish space $\X$ endowed with a distance-like 
function $d \colon \X \times \X \to [0,1]$. The function $d$ is said to be contracting for $\CP$ if there exists
$\alpha < 1$ such that the bound
\begin{equ}[e:contractd]
d(\CP(x,\cdot\,),\CP(y,\cdot\,)) \le \alpha d(x,y)
\end{equ}
holds for every pair $x,y \in \X$ with $d(x,y) < 1$.
\end{definition}

\begin{remark}
The assumption that $d$ takes values in $[0,1]$ is not a restriction at all. One can indeed check that if an unbounded function $d$ 
is contracting for a Markov operator $\CP$ and $A$ is a $d$-small set, then the same statements are true for $d$ replaced by
$d \wedge 1$.
\end{remark}
 
It may seem at first sight that \eref{e:contractd} alone is already
sufficient to guarantee the convergence of transition probabilities
toward a unique invariant measure. A little more thought shows that
this is not the case, since the total variation distance $d_\TV$ is
contracting for \textit{every} Markov semigroup.  The point here is
that \eref{e:contractd} says nothing about the pairs $(x,y)$ with
$d(x,y) = 1$, and this set may be very large.  However, combined with
the existence of a Lyapunov function $V$ that has $d$-small level
sets, it turns out that this contraction property is sufficient not
only for the existence and uniqueness of the invariant measure
$\mu_\star$, but even for having exponential convergence of transition
probabilities to $\mu_\star$ in a type of Wasserstein distance:

\begin{theorem}\label{theo:HarrisWeak}
  Let $\CP_t$ be a Markov semigroup over a Polish space $\X$ admitting
  a continuous Lyapunov function $V$.  Suppose furthermore that there
  exists $t_\star > 0$ and a distance-like function $d\colon \X\times
  \X\to[0,1]$ which is contracting for $\CP_{t_\star}$ and such that
  the level set $\{x \in \X \,:\, V(x) \le 4 \KV\}$ is $d$-small for
  $\CP_{t_\star}$. (Here $\KV$ is as in \eref{e:Lyapunov}.)

Then, $\CP_t$ can have at most one invariant probability measure $\mu_\star$. Furthermore, 
defining $\tilde d(x,y) = \sqrt{d(x,y)(1+V(x)+V(y))}$, there exists $t>0$ such that
\begin{equ}[e:contract]
\tilde d(\CP_t \mu, \CP_t \nu) \le {1\over 2} \tilde d(\mu,\nu)\;,
\end{equ} 
for all pairs of probability measures $\mu$ and $\nu$ on $\X$.
\end{theorem}

\begin{remark}
  In the special case $d = d_\TV$, we simply recover Harris' theorem,
  as stated for example in \cite{Harris}, so that this is a genuinely
  stronger statement.  It is in this sense that
  Theorem~\ref{theo:HarrisWeak} is a ``weak'' version of Harris'
  theorem where the notion of a ``small set'' has been replaced by the
  notion of a $d$-small set for a contracting distance-like function
  $d$.  The only small difference is that Harris' theorem tells us
  that the Markov semigroup $\CP_t$ exhibits a spectral gap in a total
  variation norm weighted by $1 + V$, whereas we obtain a spectral gap
  in a total variation norm weighted by $1 + \sqrt V$.  This is
  because the proof of Harris' theorem does not require the ``close to
  each other'' step (since if $d(x,y) < 1$, one has $x = y$ and the
  estimate is trivial), so that we never need to apply the
  Cauchy-Schwarz inequality.
\end{remark}

\begin{proof}[of Theorem~\ref{theo:HarrisWeak}]
Before we start the proof itself, we note that we can assume without loss of generality that 
$t_\star > \log(8\CV)/\gamma$, so that
\begin{equ}[e:boundPV8]
\CP_{t_\star} V \le {1\over 8} V + \KV\;. 
\end{equ}
This is a simple consequence of the following two facts that can be checked in
a straightforward way from the definitions:
\begin{claim}
\item If $d$ is contracting for two Markov operators $\CP$ and $\CQ$, then it is also contracting for the product 
$\CP \CQ$. (Actually it is sufficient for $d$ to be contracting for $\CP$ and to have 
\eref{e:contractd} with $\alpha = 1$ for $\CQ$.)
\item If a set $A$ is $d$-small for $\CQ$ and $d$ is contracting for $\CP$, then $A$ is also $d$-small for $\CP\CQ$.
\end{claim}
Note also that the function $\tilde d\colon \CM_1(\X)\times \CM_1(\X)\to \R_+$ is convex in each of its
arguments, so that the bound
\begin{equ}
\tilde d(\CP_t \mu , \CP_t \nu) \le \int_{\X \times\X} \tilde d(\CP_t (x,\cdot\,), \CP_t (y,\cdot\,)) \,\pi(dx,dy)\;,
\end{equ}
is valid for any coupling $\pi \in \CC(\mu,\nu)$.
As a consequence, in order to show \eref{e:contract}, it is sufficient to show that it holds in the particular case where
$\mu$ and $\nu$ are Dirac measures. In other words, it is sufficient to show that there exists $t>0$ and $\alpha' < 1$
such that
\begin{equ}[e:contract2]
\tilde d(\CP_t (x,\cdot\,), \CP_t (y,\cdot\,)) \le \alpha' \tilde d(x,y)\;,
\end{equ}
for every $x,y \in \X$. Note also that \eref{e:contract} is sufficient to conclude that $\CP_t$ can have
at most one invariant measure by the following argument. Since $V$ is a Lyapunov function for $\CP_t$,
it is integrable with respect to any invariant measure so that, if $\mu$ and $\nu$ are any two such measures,
one has $\tilde d(\mu,\nu) < \infty$. It then follows immediately from \eref{e:contract} and from the
invariance of $\mu$, $\nu$, that $\tilde d(\mu,\nu) = 0$. It follows from the lower semicontinuity of $\tilde d$
that $\mu = \nu$ as required.

In order to show that \eref{e:contract2} holds, we make use of a trick similar to the one used in
\cite{Gap}. For $\beta > 0$ a (small) parameter to be determined later,
we set
\begin{equ}
\tilde d_{\beta}(x,y) = \sqrt{d(x,y)\bigl(1 + \beta V(x) + \beta V(y)\bigr)}\;.
\end{equ}
Note that, because of the positivity of $V$, there exist constants $c$ and $C$ (depending on $\beta$ of course)
such that $c \tilde d(x,y) \le \tilde d_{\beta}(x,y) \le C \tilde d(x,y)$. As a consequence, if we can show
 \eref{e:contract2} for $\tilde d_\beta$ with some value of the parameter $\beta$, then it also holds for
 $\tilde d$ by possibly considering a larger time $t$. Just as in \cite{Gap}, we now proceed by showing that 
$\beta$ can be tuned in such a way that
\eref{e:contract2} holds, whether $x$ and $y$ are ``close to each other,'' ``far from the origin'' or ``close to the origin.''

\noindent\textbf{Close to each other.} This is the situation where $d(x,y) < 1$, so that 
\begin{equ}
\tilde d_{\beta}^2(x,y) = d(x,y)\bigl(1 + \beta V(x) + \beta V(y)\bigr)\;.
\end{equ}
In this situation, we make use of the contractivity of $d$, the fact that $V$ is a Lyapunov function,
and the Cauchy-Schwarz inequality to obtain
\begin{equs}
\bigl(\tilde d_{\beta}(\CP_{t_\star} (x,\cdot\,), \CP_{t_\star} (y,\cdot\,))\bigr)^2 &\le \inf_{\pi}
\int d(x',y') \pi(\dd x',\dd y')\, \int \bigl(1 + \beta V(x')+ \beta V(y')\bigr) \pi(\dd x',\dd y') \\
&\le \alpha d(x,y) \bigl(1 + \textstyle{\beta \over 8} (V(x) + V(y)) + 2\beta \KV\bigr) \;,
\end{equs}
where the infimum runs over all $\pi \in \CC(\CP_{t_\star} (x,\cdot\,), \CP_{t_\star} (y,\cdot\,))$.
For any given $\alpha_1 \in(\alpha,1)$, we can furthermore choose $\beta$ sufficiently small such that $\alpha(1+2\beta \KV) \le 
\alpha_1$, so that
\begin{equ}
\bigl(\tilde d_{\beta}(\CP_{t_\star} (x,\cdot\,), \CP_{t_\star} (y,\cdot\,))\bigr)^2 \le \alpha_1 \tilde d_{\beta}^2(x,y)\;.
\end{equ} 

\noindent\textbf{Far from the origin.} This is the situation where $d(x,y) \ge 1$ and $V(x) + V(y) \ge 4\KV$,
so that 
\begin{equ}
\tilde d_{\beta}^2(x,y) = 1 + \beta \bigl(V(x) + V(y)\bigr) \ge 1+3\beta \KV +  {\beta \over 4} \bigl(V(x) + V(y)\bigr)\;.
\end{equ}
Using the Lyapunov structure \eref{e:Lyapunov}, we thus get
\begin{equs}
\bigl(\tilde d_{\beta}(\CP_{t_\star} (x,\cdot\,), \CP_{t_\star} (y,\cdot\,))\bigr)^2 &\le 1  + 2\beta \KV + \CV\beta \e^{-\gamma t_\star} (V(x) + V(y))\le 1  + 2\beta \KV + {\beta \over 8} (V(x) + V(y)) \\
&\le \max \Bigl\{{1  + 2\beta \KV \over 1+3\beta \KV}, {1\over 2}\Bigr\}\tilde d_{\beta}^2(x,y) =: \alpha_2 \tilde d_{\beta}^2(x,y)\;,
\end{equs}
where we made use again of \eref{e:boundPV8}.
While $\alpha_2$ depends on the choice of $\beta$,
we see that for any fixed $\beta>0$, one has $\alpha_2 < 1$.

\noindent\textbf{Close to the origin.} This is the final situation where $d(x,y) = 1$ and $V(x) + V(y) \le 4\KV$, so that $\tilde d(x,y) \ge 1$.
In this case, we make use of the fact that the level set $\{x\,:\, V(x) \le 4\KV\}$ is assumed to be small 
to conclude that there exists a coupling $\pi$ for $\CP_{t_\star} (x,\cdot\,)$ and $\CP_{t_\star} (y,\cdot\,)$  and a constant $\eps > 0$
such that $\int d\,\dd \pi \le 1-\eps$, so that 
\begin{equs}
\bigl(\tilde d_{\beta}(\CP_{t_\star} (x,\cdot\,), \CP_{t_\star} (y,\cdot\,))\bigr)^2 &\le 
\int d(x',y') \pi(\dd x',\dd y')\, \int \bigl(1 + \beta V(x')+ \beta V(y')\bigr) \pi(\dd x',\dd y') \\
&\le (1-\eps) \bigl(1 + 2\beta \KV + 2 \beta \CV \e^{-\gamma t_\star} \bigr)\le (1-\eps) \bigl(1 + 4\beta \KV\bigr) \tilde d(x,y) \;,
\end{equs}
where we made again use of \eref{e:boundPV8}. Here, $\eps> 0$ is independent of $\beta$.
Therefore, choosing $\beta$ sufficiently small (for example $\beta = \eps / (4\KV)$), we can again make sure that 
the constant appearing in this expression is strictly smaller than $1$, thus concluding the proof of Theorem~\ref{theo:HarrisWeak}.
\end{proof}

\begin{remark}
If the assumptions of the theorem hold uniformly for $t_\star$ belonging to an open interval of times, then one
can check that Theorem~\ref{theo:HarrisWeak} implies that there exists $r > 0$ and $t_0 > 0$ such that the bound
\begin{equ}
\tilde d(\CP_t \mu, \CP_t \nu) \le e^{-r t} \tilde d(\mu,\nu)\;,
\end{equ}
holds for all $t > t_0$, instead of multiples of $t_\star$ only.
\end{remark}

If $d$ is somewhat comparable to a metric, it turns out that we can even infer the
\textit{existence} of an invariant measure from the assumptions of Theorem~\ref{theo:HarrisWeak}, just
like in the case of Harris' theorem:

\begin{corollary}
If there exists a complete metric $d_0$ on $\X$ such that $d_0 \le \sqrt{d}$ and such that $\CP_t$ is Feller on $\X$, then under the assumptions of 
Theorem~\ref{theo:HarrisWeak}, there exists a unique invariant measure $\mu_\star$ for $\CP_t$.
\end{corollary}

\begin{proof}
It only remains to show that an invariant measure exists for $\CP_t$. Fix an arbitrary probability measure
$\mu$ on $\X$ such that $\int V\,d\mu < \infty$ and let $t$ be the time obtained from Theorem~\ref{theo:HarrisWeak}. 
Since $\tilde d \ge \sqrt d \ge d_0$ by assumption and since $\tilde d(\mu, \CP_t \mu) < \infty$
by \eref{e:Lyapunov},
it then follows from \eref{e:contract2} that
\begin{equ}
d_0(\CP_{nt} \mu, \CP_{(n+1)t} \mu) \le {\tilde d(\CP_{nt} \mu, \CP_{(n+1)t} \mu)} \le {\tilde d(\mu, \CP_t \mu)\over 2^n}\;,
\end{equ}
so that the sequence $\{\CP_{nt}\mu\}_{n \ge 0}$ is Cauchy in the space of
probability measures on $\X$ endowed with the Wasserstein-$1$ distance associated to $d_0$.
Since this space is complete \cite{Villani}, there exists $\mu_\infty$ such that $\CP_{nt}\mu \to \mu_\infty$ weakly. 
In particular, the Feller property of $\CP_t$ implies $\CP_t \mu_\infty = \mu_\infty$ so that, defining $\mu_\star$ by
\begin{equ}
\mu_\star(A) = {1\over t} \int_0^t \bigl(\CP_s \mu_\infty\bigr)(A)\,\dd s\;,
\end{equ}
one can check that $\CP_r \mu_\star = \mu_\star$ for every $r>0$ as required.
\end{proof}

One standard way of using a ``spectral gap'' result like Theorem~\ref{theo:HarrisWeak} is to 
obtain the stability of the invariant measure with respect to small perturbations of the dynamic. 
Assume for the sake of the argument that $\tilde d$ satisfies the triangle inequality (in general it doesn't; see
below) and that we have a sequence of ``approximating semigroups'' $\CP_t^\eps$ such that the bound
\begin{equ}
\tilde d\bigl(\CP_t^\eps(x,\cdot), \CP_t (x,\cdot)\bigr) \le \eps C(t) \tilde V(x)\;,
\end{equ}
holds, where $C$ is a function that is bounded on bounded subsets of $\R$ and $\tilde V$ is some
positive function $\tilde V\colon \X \to \R_+$. 

Let now $\mu_\star$ denote the
invariant measure for $\CP_t$ and $\mu_\star^\eps$ an invariant measure for $\CP_t^\eps$ (which need
not be unique). Choosing $t$ as in \eref{e:contract}, one then has the bound
\begin{equs}
\tilde d(\mu_\star,\mu_\star^\eps) &= \tilde d(\CP_t \mu_\star,\CP_t^\eps \mu_\star^\eps) 
\le\tilde d(\CP_t \mu_\star,\CP_t \mu_\star^\eps) + \tilde d(\CP_t \mu_\star^\eps,\CP_t^\eps \mu_\star^\eps) \\
&\le {1\over 2} \tilde d(\mu_\star,\mu_\star^\eps) + \eps C(t) \int_\X \tilde V(x)\,\mu_\star^\eps(\dd x)\;,
\end{equs}
from which we deduce that $\tilde d(\mu_\star,\mu_\star^\eps) \le 2 \eps C(t) \int_\X \tilde V(x)\,\mu_\star^\eps(\dd x)$.
If one can obtain an \textit{a priori} bound on $\mu_\star^\eps$ that ensures that $\int_\X \tilde V(x)\,\mu_\star^\eps(\dd x)$
is bounded independently of $\eps$, this shows that the distance between the invariant measures for
$\CP_t$ and $\CP_t^\eps$ is comparable to the distance between the transition probabilities for some fixed
time $t$.

This argument is still valid if the distance function $\tilde d$ satisfies a weak form of the triangle inequality,
\ie if there exists a positive constant $K>0$ such that
\begin{equ}[e:almosttriangle]
\tilde d(x,y) \le K \bigl(\tilde d(x,z) + \tilde d(z,y)\bigr)\;,
\end{equ}
for every $x,y,z \in \X$. This turns out to be often satisfied in practice, due to the following result:

\begin{lemma}\label{lem:triangleweak}
Let $d\colon \X\times \X\to [0,1]$ be a distance-like function and assume that there exists a constant $K_d$
such that 
\begin{equ}[e:atd]
d(x,y) \le K_d \bigl(d(x,z) + d(z,y)\bigr)\;,
\end{equ}
holds for every $x,y,z \in \X$. Assume furthermore that  $V \colon \X \to \R_+$ is such that
there exist constants $c$, $C$ such that the implication
\begin{equ}[e:propV]
d(x,z) \le c \qquad\Rightarrow\qquad V(x) \le C V(z)
\end{equ}
holds. Then, there exists a constant $K$ such that \eref{e:almosttriangle} holds for $\tilde d$ defined
as in Theorem~\ref{theo:HarrisWeak}.
\end{lemma}

\begin{proof}
Note first that it is sufficient to show that there exists a constant $K$ such that
\begin{equ}[e:wantedbound]
d(x,y) \bigl(1+V(x)+V(y)\bigr) \le K \bigl(d(x,z) \bigl(1+V(x)+V(z)\bigr) + d(z,y) \bigl(1+V(z)+V(y)\bigr)\bigr)\;.
\end{equ}
Since $d$ is symmetric, we can assume without loss of generality that $V(x) \ge V(y)$. We consider the
following two cases

If $d(x,z) \ge c$, then the boundedness of $d$ implies the existence of a constant $\tilde C$ such that
\begin{equ}
d(x,y) \bigl(1+V(x)+V(y)\bigr) \le \tilde C\bigl(1+V(x)+V(y)\bigr)
\le {\tilde C\over c} d(x,z)\bigl(1+V(x)+V(y)\bigr) \le {2\tilde C\over c} d(x,z)\bigl(1+V(x)\bigr)\;,
\end{equ}
from which \eref{e:wantedbound} follows with $K = 2\tilde C/c$.

If $d(x,z) \le c$, we make use of our assumptions \eref{e:atd} and \eref{e:propV} to deduce that
\begin{equs}
d(x,y) \bigl(1+V(x)+V(y)\bigr) &\le K_d \bigl(d(x,z) + d(z,y)\bigr)\bigl(1+V(x)+V(y)\bigr) \\
&\le 2K_d d(x,z) \bigl(1+V(x)\bigr) + 2CK_d d(z,y) \bigl(1+V(z)\bigr)\;,
\end{equs}
from which \eref{e:wantedbound} follows with $K = 2K_d(1\vee C)$.
\end{proof}

\begin{remark}
If $\X$ is a Banach space and $d(x,y) = 1 \wedge \|x-y\|$, then Lemma~\ref{lem:triangleweak} essentially states that
$\tilde d$ satisfies \eref{e:almosttriangle} provided that $V(x)$ grows at most exponentially with $\|x\|$.
\end{remark}

The following result (which we already used in the previous section) relates the contraction property to the conditions 
in our main result on the convergence of transition probabilities.

\begin{proposition}\label{prop:contr}
Let $(\cP_t)$ be a Feller Markov semigroup on $\X$  and assume that there exists a continuous metric $d$ which generates the topology of $\X$ and 
which is contracting for $\cP_t$ for some $t>0$. Then the second condition in Theorem~\ref{thm:conv} is satisfied for the Markov kernel
$\cP_t$. 
\end{proposition}

Before we turn to the proof of Proposition~\ref{prop:contr}, we give the following result which is essential to settle the measurability questions
arising in the proof:

\begin{lemma}\label{lem:measure}
Let $\CQ$ be a Feller Markov operator on a Polish space $\X$ and let $d$ be a $[0,1]$-valued distance-like function on $\X \times \X$ which is contracting for $\CQ$. 
Then there exists $\tilde \alpha < 1$ and a Markov operator $\cT$ on $\X \times \X$ such that transition probabilities of $\cT$ are couplings of
the transition probabilities for $\CQ$ and such that the inequality
\begin{equ}
\bigl(\cT d\bigr)(x,y) \le \tilde \alpha d(x,y) \;,
\end{equ}
holds for every $(x,y)$ such that $d(x,y) < 1$.
\end{lemma}

\begin{proof}
Denote as before by $\CM(\X \times \X)$ the set of probability measures on $\X \times \X$ endowed with the topology of weak convergence, so that it is again
a Polish space. Let $\tilde \alpha \in (\alpha, 1)$, where $\alpha$ is as in Definition~\ref{def:contr}. For every $(x,y) \in \X \times \X$, denote
\begin{equ}
F(x,y) =  
	\{ \Gamma \in \C\bigl(\CQ(x,\cdot\,),\CQ(y,\cdot\,)\bigr)\,:\, \Gamma(d) \le \tilde \alpha d(x,y)\} \;,
\end{equ}
and denote by $\Delta$ the closure in $\X \times \X$ of the set $\{d(x,y) < 1\}$. We know that $F(x,y)$ is non-empty by assumption whenever $d(x,y) < 1$.
The Feller property of $\CQ$ then ensures that this is also true for $(x,y) \in \Delta$.

The proof of the statement is complete as soon as we can show that there exists a \textit{measurable} map $\hat \cT \colon \X \times \X \to \cM(\X \times \X)$ such that 
$\hat \cT(x,y) \in F(x,y)$ for every $x,y \in \X$, since it then suffices to set for example
\begin{equ}
\cT(x,y;\cdot\,) = 
\left\{\begin{array}{cl}
	 \cT(x,y) & \text{if $(x,y) \in \Delta$,} \\
	\CQ(x,\cdot\,) \otimes \CQ(y,\cdot\,) & \text{otherwise.}
\end{array}\right.
\end{equ}
Since the set $F(x,y)$ is closed for every pair $(x,y)$ by the continuity of $d$, 
this follows from the Kuratowski, Ryll-Nardzewski selection theorem \cite{KR65,Wagner} provided we can show that, for every open set $U \subset \cM(\X \times \X)$,
the set $F^{-1}(U) = \{(x,y) \,:\, F(x,y) \cap U \neq \emptyset\}$ is measurable.

Since on a Polish space every open set is a countable union of closed sets and since $F^{-1}(U \cup V) = F^{-1}(U) \cup F^{-1}(V)$ (the same is not true 
in general for intersections!), the claim follows if we can show that $F^{-1}(U)$ is measurable for every closed set $U$. 
Under our assumptions, $F^{-1}(U)$ actually turns out to be closed if $U$ is closed. To see this, take a convergent sequence $(x_n, y_n) \in F^{-1}(U)$.
The definition of $F$ implies that there exist couplings $\Gamma_n \in \C\bigl(\CQ(x_n,\cdot\,),\CQ(y_n,\cdot\,)\bigr)$ with $\Gamma_n(d) \le \tilde \alpha d(x_n, y_n)$.
Since $\CQ$ is Feller, the sequence $\{\Gamma_n\}$ is tight, so that there exists a subsequence converging to a limit $\Gamma$. Since $\Gamma$ belongs
to $\C\bigl(\CQ(x,\cdot\,),\CQ(y,\cdot\,)\bigr)$ and since, by the continuity of $d$, we have $\Gamma(d) \le \tilde \alpha d(x,y)$, $(x,y) \in F^{-1}(U)$ as claimed.
\end{proof}

\begin{proof}[of Proposition~\ref{prop:contr}]
By assumption, there exists some $\alpha \in (0,1)$ such that $d (\cP_t(x,.),\cP_t(y,.)) \leq  \alpha d(x,y)$ for all
$x,y \in \X$ which satisfy $d(x,y)<1$. 
Consider the Markov operator $\cQ:=\cP_t$. 
Let $\cT$ be the Markov operator from Lemma~\ref{lem:measure} so that
if, for fixed $x,y \in \X$, we denote the corresponding chain starting at $(x,y)$ by $(X_n,Y_n)$, then we have 
$\E d(X_1,Y_1) \le \tilde\alpha d(x,y)$ whenever $d(x,y)<1$. Let $\Gamma_{x,y}$ be the law of $(X_n,Y_n),n \in \N_0$. Now we define
$$
V_n:=\tilde\alpha^{-n} d(X_n,Y_n)
$$
and $\tau:= \inf \{n \in \N: d(X_n,Y_n) \ge 1\}$.
Then $V_{n \wedge \tau}$, $n \in \N_0$ is a non-negative supermartingale and therefore
$$
\P\{\tau < \infty\} \le \P\{ \sup_{n \ge 0} V_{n \wedge \tau} \ge 1 \} \le d(x,y)\;,
$$
i.e. 
$$
\Gamma_{x,y} \{d(X_n,Y_n) \le \tilde\alpha^n \mbox{ for all } n \in \N_0\} \ge 1-d(x,y)\;.
$$
This shows that the second assumption in Theorem~\ref{thm:conv} is satisfied for the chain associated to $\CQ$.
\end{proof}

\section{Application of the spectral gap result to SDDEs}
\label{sec:applgap}
In this section, we apply the abstract results from the previous section to the 
problem of exponential convergence to an invariant measure for the type of stochastic delay equations
considered earlier. The main problem will turn out to be to find a distance-like function $d$ which is contracting.
In order to obtain an exponential convergence result, we will have to assume, just like in the case of Harris
chains \cite{MT,Harris} some Lyapunov structure. We therefore introduce the following assumption:

\begin{assumption}
There exists a continuous function $V \colon \CC \to \R_+$ such that $\lim_{\|X\| \to \infty} V(X) = +\infty$ and such that
there exist strictly positive constants $\CV$, $\gamma$ and $\KV$ such that the bound
\begin{equ}
\E V(X_t) \le \CV e^{-\gamma t} V(X_0) + \KV\;,
\end{equ}
holds for solutions to \eref{e:sdde2} with arbitrary initial conditions $X_0 \in \CC$. 
\end{assumption}

The distance-like function $d$ that we are going to use in this section is given by
\begin{equ}[e:defd]
d(X,Y) = 1 \wedge \delta^{-1} \|X - Y\|\;,
\end{equ}
for a suitable (small) constant $\delta$ to be determined later. We start by verifying that bounded sets are $d$-small
for every value of $\delta$ and we will then proceed to showing that under suitable assumptions,
it is possible to find $\delta > 0$ such that $d$ is also contracting.

\subsection{Bounded sets are $\mathbf{\it d}$-small}\label{sec:contrSDDE}

\begin{proposition}
Let the assumptions of Theorem \ref{thm:weakErgodicity} be satisfied, let $d$ be as in \eref{e:defd} and let $t \geq 2r$
be arbitrary. Then every bounded set is $d$-small for $\CP_t$.
\end{proposition}

\begin{proof}
Fix $t \geq 2r$. We show that every closed ball $\mathrm{B}_R \subset \C$ with center 0 and radius $R$ is $d$-small for $\cP_t$.
By Lemma \ref{lem:support}, we know that $p:=\inf_{x \in \mathrm{B}_R} \cP_t\big( x,\mathrm{B}_{\delta/4} \big) >0$. 
Let $x,y \in \mathrm{B}_R$ and let $X$ and $Y$ be solutions of \eref{e:sdde2} with initial conditions $x$ and $y$ respectively. 
We couple $X$ and $Y$ independently. Then
\begin{align*}
d(\cP_t(x,.),\cP_t(y,.)) &\leq \E\big(1 \wedge \big(\delta^{-1}\|X_t-Y_t\|\big)\big)\\  
&\le \P\big( \big\{X_t \notin 
\mathrm{B}_{\delta/4} \big\} \cup \big\{Y_t \notin \mathrm{B}_{\delta/4} \big\}\big)
+ \frac 12 \P\big\{ X_t \in \mathrm{B}_{\delta/4},Y_t \in \mathrm{B}_{\delta/4}\big\} \\
&\le 1-\frac 12 p^2
\end{align*}
for all $x,y \in \mathrm{B}_R$, so the claim follows.
\end{proof}

\subsection{The distance $\mathbf{\it d}$ is contracting}

Before we start, we give the following a priori estimate that shows that trajectories of \eref{e:sdde2} driven by 
\textit{the same} realisation of the noise cannot separate too rapidly. More precisely, we have:

\begin{proposition}\label{prop:apriori}
Let the assumptions of Theorem \ref{thm:weakErgodicity} be satisfied. There exists $\kappa > 0$ such that the bound
\begin{equ}
\E \|X_t - \tilde X_t\|^4 \le \e^{\kappa (1+t)^2} \|X_0 - \tilde X_0\|^4\;,
\end{equ}
holds for all $t \ge 0$ and any pair of initial conditions $X_0, \tilde X_0 \in \CC$.
\end{proposition} 

\begin{proof}
The proof is similar to the argument used in the proof of Theorem~\ref{thm:weakErgodicity}. Setting $Z(t) = X(t) - \tilde X(t)$, 
we have the bound
\begin{equ}
    \dd |Z(t)|^2 = 2\scal{f(X_t)- f(\tilde X_t),Z(t)}\,\dd t +
    \Norm{g(X_t)-g(\tilde X_t)}^2 \,\dd t + \dd M(t) \leq K \|Z_t\|^2\,\dd t  + \dd M(t) \;,
\end{equ}
where $M$ is a martingale with quadratic variation process bounded by $C\int_0^t \|Z_s\|^4\,\dd s$.
Defining $M^*(t) = \sup_{s \le t} M(s)$, we thus obtain the bound
\begin{equ}
\|Z_t\|^2 \le \|Z_0\|^2 + K \int_0^t \|Z_s\|^2\,\dd s + M^*(t)\;,
\end{equ}
so that 
\begin{equs}
\E \|Z_t\|^4 &\le 3 \E \Bigl(\|Z_0\|^4 + K^2 \Bigl(\int_0^t \|Z_s\|^2\,\dd s\Bigr)^2 + \bigl(M^*(t)\bigr)^2\Bigr)\\
&\le 3 \Bigl(\|Z_0\|^4 + K^2 t \int_0^t \E \|Z_s\|^4\,\dd s + C \int_0^t \E \|Z_s\|^4\,\dd s\Bigr)\;,
\end{equs}
where we used the Burkholder-Davis-Gundy inequality \cite{RY} in order to bound the expectation of $(M^*)^2$.
The claim follows from Gronwall's lemma.
\end{proof}

In this subsection, we show one possible way of verifying that $d$ is
contracting that is suited to our problem.  This is by far not the
only one. One can check for example that the procedure followed in
\cite{Gap} allows to construct a contracting distance for the
degenerate 2D stochastic Navier-Stokes equations by using a gradient
bound on the semigroup.  A general version of this argument is
presented in Section~\ref{sec:contractSPDE} below.  For the problem at
hand, it seems however more appropriate and technically
straightforward to consider a ``binding construction'' in the
terminology of \cite{MasYou,H,MatNS}.

We fix two initial conditions $X_0, \tilde X_0 \in \C$ and consider the construction from Section~\ref{sec:uniqueness}. We fix some $\gamma_0 > 0$ and
choose $\lambda$ sufficiently large so that  the conclusion of Lemma~\ref{lem:contraction} holds.
As in the proof of Theorem~\ref{thm:weakErgodicity},
we also introduce the stopping time $\tau = \inf \{t > 0 \,:\, \int_0^t |v(s)|^2\, \dd s \ge \eps^{-1} \|X_0 - \tilde X_0\|^2\}$, where $v$ is as in the 
proof of Theorem~\ref{thm:weakErgodicity}. (Note that the value of $\eps$ is not necessarily that from Section~\ref{sec:uniqueness}, but will be determined later.)
We also define
$\tilde v$ by $\tilde v(s) = v(s) \one_{\tau > s}$.

This defines a map $\Psi$ from $\Omega:=\CC([0,\infty),\R^m)$ to itself by $\Psi(w) = w + \int_0^\cdot \tilde v(s)\,\dd s$ (the map $\Psi$ furthermore depends
on the initial conditions $X_0$ and $\tilde X_0$, but we suppress this dependence from the notation). The image $\tilde \P$
of Wiener measure $\P$ under $\Psi$ has a density $\CD(w) = d\tilde \P / d\P$.

The aim of introducing the cutoff is that if we define $\tilde \CD(w) = 1/\CD(w)$,
we obtain ``for free'' bounds of the type
\begin{equ}
\int (1-\CD(w))^2\,\P(\dd w) \le C\eps^{-1} \|X_0 - \tilde X_0\|^2\;,\qquad
\int (1-\tilde \CD(w))^2\,\tilde \P(\dd w) \le C\eps^{-1} \|X_0 - \tilde X_0\|^2\;,
\end{equ}
for some constant $C>0$, provided that we restrict ourselves to pairs of initial conditions such that
\begin{equ}[e:boundX0Y0]
\|X_0 - \tilde X_0\|^2 \le \eps\;.
\end{equ}
Had we not introduced the cut-off, we would need to get exponential integrability of $v$ first.

The map $\Psi$ allows to construct, for any two initial conditions $X_0$ and $\tilde X_0$, 
a coupling for $\P$ with itself in the following way. Define the map $\tilde \Psi \colon \Omega \to \Omega\times \Omega$
by $\tilde \Psi(w) = (w,\Psi(w))$, denote by $\pi^i$ the projection onto the $i$th component of $\Omega \times \Omega$,
and set
\begin{equs}
\Pi_0(\dd w_1,\dd w_2) &= (1\wedge \tilde \CD(w_2)) \bigl(\tilde \Psi\push\P\bigr)(\dd w_1, \dd w_2)\;, \\
\Pi(\dd w_1, \dd w_2) &= \Pi_0(\dd w_1, \dd w_2) + Z^{-1} (\P - \pi^1\push\Pi_0)(\dd w_1)(\P - \pi^2\push\Pi_0)(\dd w_2)\;,
\end{equs}
\begin{wrapfigure}{r}{4cm}
\includegraphics{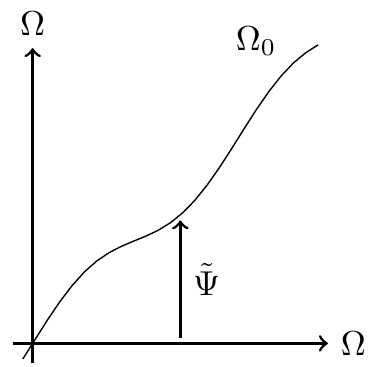}
\end{wrapfigure}
where $Z = 1 - \Pi_0(\Omega\times \Omega) = {1\over 2}\|\P - \tilde \P\|_{\TV}$ is a suitable constant. One can check that $\Pi$
as defined above is a coupling for $\P$ and $\P$. Furthermore, it is designed in such a way that it maximises
the mass of the set $\Omega_0 = \{(w,w')\,:\, w' = \Psi(w)\}$. We claim that this coupling is designed in such a way 
that its image under the product solution map of \eref{e:sdde2} allows to verify that $d$ as in \eref{e:defd} is contracting 
for some sufficiently small value of $\delta > 0$ to be determined later. Note that, since 
the bound \eref{e:contract} only needs to be checked for pairs of initial conditions with $d(X_0,\tilde X_0) < 1$, the constraint
\eref{e:boundX0Y0} is satisfied provided that we make sure that $\delta^2 \le \eps$.

In order to see that this is indeed the case, we fix a terminal time $t$ and we break the space 
$\Omega \times \Omega = \Omega_1 \cup \Omega_2 \cup \Omega_3$ into three parts:
\begin{equs}
\Omega_1 &= \bigl\{(w,w')\,:\, w' = \Psi(w)\;\&\; \tau(w) \ge t\bigr\} \;,\\
\Omega_2 &= \bigl\{(w,w')\,:\, w' = \Psi(w)\;\&\; \tau(w) < t \bigr\} \;,\\
\Omega_3 &= \bigl\{(w,w')\,:\, w' \neq \Psi(w)\bigr\} \;.
\end{equs}
Here, we made use of the stopping time $\tau$ defined at the beginning of this section.
We consider the set $\Omega_1$ as being the set of ``good'' realisations and we will show that $\Omega_1$ has
high probability. The contributions from the other two sets will be considered as error terms.

Consider now a pair $(X_0, \tilde X_0)$ of initial conditions such that $d(X_0, \tilde X_0) < 1$, which is to say that
$\|X_0 - \tilde X_0\| < \delta$. Denote by $X_t$ and $\tilde X_t$ the solutions to \eref{e:sdde2} driven by the noise realisations
$w$ and $w'$ respectively. We then have
\begin{equs}
\int_{\Omega_1} d(X_t(w), \tilde X_t(w'))\, \Pi(\dd w, \dd w') &\le \delta^{-1} \int_{\Omega_1} \|X_t(w) - \tilde X_t(w')\|\, \Pi(\dd w, \dd w') \\
&\le  \delta^{-1} \int_{\Omega} \|X_t(w) - \tilde X_t(\Psi(w))\|\, \P(\dd w) \le C\delta^{-1} \e^{-\gamma_0 t} \|X_0 - \tilde X_0\| \\
&= C\e^{-\gamma_0 t} d(X_0, \tilde X_0)\;,
\end{equs}
where we made use of the bounds obtained in Lemma \ref{lem:contraction}. Regarding the integral over $\Omega_2$,
we combine Lemma \ref{lem:contraction}, Proposition~\ref{prop:apriori}  and the strong Markov property
to conclude that
\begin{equs}
\int_{\Omega_2} d(X_t(w), \tilde X_t(w'))\, \Pi(\dd w, \dd w') &\le \delta^{-1} \int_{\tau < t} \|X_t(w) - \tilde X_t(\Psi(w))\|\, \P(\dd w) \\
&\le \delta^{-1} \Bigl(\int_{\tau < t} \|X_t(w) - \tilde X_t(\Psi(w))\|^2\, \P(\dd w)\Bigr)^{1/2} \sqrt{\P(\tau < t)}\\
&\le \delta^{-1} \E \Bigl(C \e^{-\gamma_0 \tau} \e^{\kappa (1+t-\tau)^2}\Bigr)\,\|X_0 - \tilde X_0\| \sqrt{\P(\tau < t)}\\\
&\le C\e^{\kappa (1+t)^2} d(X_0, \tilde X_0)\sqrt{\P(\tau < t)}\;.
\end{equs}
At this stage, we combine Lemma \ref{lem:contraction} with Chebychev to conclude that 
\begin{equ}
\P(\tau < t) \le \P \Bigl(\int_0^\infty |v(s)|^2\, \dd s \ge \eps^{-1} \|X_0 - \tilde X_0\|^2\Bigr) \le C\eps\;,
\end{equ}
for some constant $C$ independent of $t$ and the pair $(X_0, \tilde X_0)$. Finally, we obtain the bound
\begin{equs}
\int_{\Omega_3} d(X_t(w), \tilde X_t(w'))\, \Pi(\dd w, \dd w') &\le \Pi(\Omega_3) = \int_\Omega (1 - 1\wedge \tilde D(w))\,\tilde \P(\dd w) \\
&= \int_\Omega \bigl(0 \vee (1-\tilde D(w)) \bigr)\,\tilde \P(\dd w) \le \Bigl(\int_\Omega (1-\tilde D(w))^2\,\tilde \P(\dd w) \Bigr)^{1/2} \\
&\le C\eps^{-1/2} \|X_0 - \tilde X_0\| \le C\delta  \eps^{-1/2}  d(X_0, \tilde X_0)\;.
\end{equs}
The required bound follows by first taking $\eps$ small enough and then taking $\delta$ small enough.

\subsection{Construction of contracting distances for SPDEs}
\label{sec:contractSPDE}

Finally, we want to show how the existence of a contracting distance for a Markov semigroup
$\CP_t$ can be verified in the case of stochastic PDEs. This is very similar to the calculation performed in \cite{Gap},
but it has the advantage of not being specific to the Navier-Stokes equations. Recall that \cite{Ergodic} yields
conditions under which the Markov semigroup (over some separable Hilbert space $\CH$) 
generated by a class of stochastic PDEs satisfies the following gradient bound
for every function $\phi \in \CC^1(\CH, \R)$:
\begin{equ}[e:gradient]
\|D \CP_t \phi(X)\| \le W(X) \Bigl(\e^{-\tilde \gamma t} \sqrt{\bigl(\CP_t \|D\phi\|^2\bigr)(X)} + C\|\phi\|_\infty \Bigr)\;.
\end{equ}
Here, $W \colon \CH \to \R_+$ is some continuous function that controls the regularising properties of $\CP_t$ and
$\tilde \gamma, C$ are some strictly positive constants. It turns out that if the semigroup $\CP_t$ has sufficiently good
dissipativity properties with respect to $W$, then one can find a contracting distance function for it. 

Before we state the result, let us define a family of ``weighted metrics'' $\rho_p$ on $\CH$ by
\begin{equ}
\rho_p(X,Y) = \inf_{\gamma \colon X \to Y} \int_0^1 W^p(\gamma(t))\,\|\dot\gamma(t)\|\,\dd t\;,
\end{equ}
where the infimum runs over all Lipschitz continuous paths $\gamma \colon [0,1] \to \CH$ connecting $X$ to $Y$.
With this notation at hand, we have:

\begin{proposition}
Let $\{\CP_t\}_{t \ge 0}$ be a Markov semigroup over a separable Hilbert space $\CH$ satisfying the bound \eref{e:gradient}
for some continuous function $W\colon \CH \to [1,\infty)$. Assume furthermore that there exists $p > 1$, a time $t_\star > 0$ and
a constant $\tilde C > 0$ such that the bound
\begin{equ}[e:superLyap]
\CP_t W^{2p} \le \tilde C W^{2p-2}\;,
\end{equ}
holds for every $t \ge t_\star$. (In other words, $W$ is a kind of super-Lyapunov function for $\CP_t$.) Then, there exists
$\delta > 0$ and $T>0$ such that the metric $d(X,Y) = 1 \wedge \delta^{-1} \rho_p(X,Y)$ is contracting for $\CP_T$.
\end{proposition}

\begin{proof}
By Monge-Kantorowitch duality, it is sufficient to show that there exist $T>0$ and $\delta>0$ such that the bound
\begin{equ}[e:wanted]
|\CP_T \phi(X) - \CP_T \phi(Y)| \le {\rho_p(X,Y) \over 2\delta}\;,
\end{equ} 
holds for every $\CC^1$ function $\phi\colon \CH \to \R$ which has Lipschitz constant $1$ with respect to $d$.
Note now that such a function $\phi$ satisfies
\begin{equ}
|\phi(X)| \le {1\over 2}\;,\qquad \|D\phi(X)\| \le {W^p(X)\over \delta}\;.
\end{equ}
In particular, it follows from the gradient bound \eref{e:gradient} combined with \eref{e:superLyap} that
for $T \ge t_\star$, one has
\begin{equ}
\|D\CP_T \phi(X)\| \le W(X) \Bigl(\delta^{-1} e^{-\tilde \gamma T} \sqrt{\tilde C} W^{p-1}(X) + {C\over 2} \Bigr)\;.
\end{equ}
Choosing $T$ sufficiently large and $\delta$ sufficiently small, we see that it is possible to ensure that
\begin{equ}
\|D\CP_T \phi(X)\| \le {W^p (X) \over 2 \delta}\;.
\end{equ}
Since, on the other hand, for any path $\gamma$ connecting $X$ to $Y$ we have
\begin{equ}
|\CP_T \phi(X) - \CP_T \phi(Y)| \le \int_0^1 \|D\CP_T\phi(\gamma(s))\| \,\|\dot \gamma(s)\|\,\dd s\;,
\end{equ}
the requested bound \eref{e:wanted} follows at once.
\end{proof}

\end{document}